\documentclass{amsart}

\usepackage{lineno,hyperref,amsmath, amssymb,amsthm}
\newtheorem{theorem}{Theorem}[section]
\newtheorem{lemma}[theorem]{Lemma}
\newtheorem{proposition}[theorem]{Proposition}
\newtheorem{corollary}[theorem]{Corollary}

\theoremstyle{definition}

\newtheorem{example}[theorem]{Example}

\theoremstyle{remark}
\newtheorem{remark}[theorem]{Remark}

\numberwithin{equation}{section}

\newcommand{\fg}{\mathfrak{g}}

\newcommand{\fk}{\mathfrak{k}}

\newcommand{\fs}{\mathfrak{s}}

\begin{document}
\pagestyle{plain}

\title{Reductions of minimal Lagrangian submanifolds with symmetries}

\author{Toru Kajigaya}%\fnref{myfootnote}}
%\address{National Institute of Technology, Kitakyushu Collage, 5-20-1 Shii, Kokuraminamiku, Kitakyushu, Fukuoka, 802-0985 Japan}
%\email{hashinaga@kc.ac.jp}
\address{National Institute of Advanced Industrial Science and Technology (AIST), MathAM-OIL, 
%Tohoku University \endgraf
Sendai 980-8577, Japan
}

%\fntext[myfootnote]{}


\email{kajigaya.tr@aist.go.jp}


\date{\today}

\begin{abstract}
 Let $M$ be a Fano manifold equipped with a K\"ahler form $\omega\in 2\pi c_1(M)$ and $K$ a connected compact Lie group acting on $M$ as holomorphic isometries. In this paper, we show the minimality of a $K$-invariant Lagrangian submanifold $L$ in $M$ with respect to a globally conformal K\"ahler metric is equivalent to the minimality of the reduced Lagrangian submanifold $L_0=L/K$ in a K\"ahler quotient $M_0$ with respect to the Hsiang-Lawson metric. Furthermore, we give some examples of K\"ahler reductions by using a circle action obtained from a cohomogenenity one action on a K\"ahler-Einstein manifold of positive Ricci curvature. Applying these results, we obtain several examples of minimal Lagrangian submanifolds via reductions.
\end{abstract}

\maketitle

\section{Introduction}\label{Sec1}
Minimal submanifolds in a Riemannian manifold are classical objects in a submanifold geometry, and investigated by many authors. In particular, the group symmetry of the ambient Riemannian manifold is a useful notion, and many examples of minimal submanifolds with symmetries have been constructed.  Let $(\overline{M}, \overline{g})$ be a  Riemannian manifold with Riemannian metric $\overline{g}$ and $K$ a connected Lie group acting on $\overline{M}$ as isometries. It was first proved by Hsiang-Lawson \cite{HL} that  a $K$-invariant submanifold $N$ of $\overline{M}$ is minimal if and only if $N$ is a stationary point of the volume functional under any compactly supported $K$-equivariant infinitesimal deformation of $N$, and the minimality of $N$ in $\overline{M}$ is equivalent to the minimality of the reduced space $N^*/K$ in the orbit space $\overline{M}^*/K$ equipped with an appropriate metric which we call the {\it Hsiang-Lawson metric} (see \cite{HL} or Subsection \ref{subsec2.1} in the present paper for the definition), where $\overline{M}^*$ (resp. $N^*$) is the set of principal $K$-orbits in $\overline{M}$ (resp. $N$).  This fundamental method can be applied to several situations and produces many examples of minimal submanifolds (see \cite{HL}).

On the other hand, Lagrangian submanifolds in a symplectic manifold $(M,\omega)$ play an important role in symplectic geometry. In the symplectic contexts, there is a well-known reduction procedure so called the Marsden-Weinstein-Meyer symplectic reduction. Suppose a Lie group $K$ acts on $M$ in a Hamiltonian way, namely, the action preserves the symplectic form $\omega$ and admits a moment map $\mu: M\rightarrow \fk^*$, where $\fk$ is the Lie algebra of $K$. If $K$ acts on a Lagrangian submanifold $L$, then $L$ is contained in $\mu^{-1}(c)$ for some $c\in \mathfrak{z}(\fk^*)$, where $\mathfrak{z}(\fk^*)$ is the center of $\fk^*$. Furthermore, if $c$ is a regular value of $\mu$ and $K$ acts on $\mu^{-1}(c)$ freely,  then the  Marsden-Weinstein-Meyer symplectic reduction yields another symplectic manifold $(M_c=\mu^{-1}(c)/K, \omega_c)$ and $L$ is reduced to a Lagrangian submanifold $L_c=L/K$ in $M_c$. In this situation,  it is natural to ask whether a certain property of $L$ is related to a property of the reduced Lagrangian submanifold $L_c$.

 If we equip $M$ with a Riemannian metric $g$, then we define notions of minimality of a Lagrangian submanifold as follows:  A Lagrangian submanifold $L$ in $M$ is called {\it minimal (resp. Hamiltonian minimal) with respect to $g$} if $L$ is a stationary point of the volume functional measured by $g$ under any infinitesimal deformation of $L$ (resp. any Hamiltonian deformation  of $L$ in the sense of \cite{Oh}). 
Minimal Lagrangian submanifolds are of particular interests in several contexts, e.g., the Lagrangian mean curvature flows and the Hamiltonian volume minimizing problem  (see \cite{KK}, \cite{LP}, \cite{Oh}, \cite{SmoWang} and references therein). 

When $M$ is a K\"ahler manifold, the symplectic quotient space $M_c$ inherits a natural K\"ahler structure, and we call this reduction procedure the K\"ahler reduction (see Section \ref{Sec2}). 
In \cite{Dong}, Dong applied  Hsiang-Lawson's method to Hamiltonian minimal Lagrangian submanifolds in a K\"ahler manifold $M$ and proved that a $K$-invariant Lagrangian submanifold $L$ in $M$ is Hamiltonian minimal with respect to the K\"ahler metric $g$ if and only if so is $L_c$ in the K\"ahler quotient $M_c$ with respect to the Hsiang-Lawson metric of $g$. By using this reduction method, Dong constructed infinitely many Hamiltonian minimal Lagrangian submanifolds with large symmetries in $\mathbb{C}P^n$ and $\mathbb{C}^n$.
On the other hand, Legendre-Rollin studied the Hamiltonian stability and rigidity of the reduced Lagrangian submanifold when $M_c$ is a compact toric K\"ahler manifold, and gave an interesting application for Lagrangian tori by using the symplectic reduction (see Section 3 in \cite{LR}). 

In the present paper, we first improve Dong's result and give a Hsiang-Lawson type theorem for minimal Lagrangian submanifolds in a K\"ahler manifold (Theorem \ref{thm2.11}). As a consequence, we show that if $M$ is a Fano manifold equipped with a K\"ahler form $\omega\in 2\pi c_1(M)$, then a $K$-invariant Lagrangian submanifold in $M$ is minimal with respect to a globally conformal K\"ahler metric $\tilde{g}$ if and only if the reduced Lagrangian submanifold in the K\"ahler quotient space is minimal with respect to the Hsiang-Lawson metric of $\tilde{g}$.   Note that the globally conformal K\"ahler metric $\tilde{g}$ is defined by the K\"ahler metric on $M$ and the Ricci form of $M$ (see Subsection \ref{subsec2.3}).  In particular, our result can be applied to any closed K\"ahler-Einstein manifold of positive Ricci curvature as ambient manifold $M$. If this is the case, then the globally conformal K\"ahler metric $\tilde{g}$ is taken as the K\"ahler-Einsiten metric of $M$. However, we remark that the Hsiang-Lawson metric is not necessarily a K\"ahler metric on the quotient space. 
See Section \ref{Sec2} for more details.

Next, we apply Theorem \ref{thm2.11} to construct new examples of minimal Lagrangian submanifolds. 
By Theorem \ref{thm2.11}, under the assumption of $M$, a $K$-invariant minimal Lagrangian submanifold in $M$ yields a minimal Lagrangian submanifold in the K\"ahler quotient space (with respect to the appropriate metrics) via the reduction, and vise-versa. However, it is not  easy to see what the K\"ahler quotient is explicitly in general. When $M$ is a Fano manifold, Futaki proved that a K\"ahler quotient space is a Fano manifold again, however, the quotient structure is not necessarily K\"ahler-Einstein (See \cite{Futaki} or Section \ref{Sec2} in the present paper). 

On the other hand, when $M$ is the complex Euclidean space $\mathbb{C}^n$,  the Hopf fibration $\pi: S^{2n-1}\rightarrow \mathbb{C}P^{n-1}$ gives a typical example of the K\"ahler reduction via the standard $S^1$-action on $\mathbb{C}^n$.  In this case, the  reduction still remains large symmetries, and the quotient space (i.e., $\mathbb{C}P^{n-1}$) becomes a compact homogeneous K\"ahler-Einstein manifold.  Moreover, we note that  some extrinsic properties of a $S^1$-invariant Lagrangian submanifold in $\mathbb{C}^n$ are closely related to properties of the reduced Lagrangian submanifold in $\mathbb{C}P^{n-1}$ (cf. \cite {AO}, \cite{Dong}, \cite{LR} and \cite{Oh}).   In Section \ref{Sec3}, we give a special case of the K\"ahler reduction, as a generalization of the Hopf fibration, in a compact K\"ahler-Einstein manifold $M$ of positive Ricci curvature by using a circle action obtained from a cohomogeneity one action on $M$ (Theorem \ref{thm3.7}).  The resulting K\"ahler quotient space is always a compact homogeneous K\"ahler-Einstein manifold of positive Ricci curvature and the Hsiang-Lawson metric coincides with the K\"ahler-Einstein metric. Thus, the geometric structure of the quotient space may be well-understood.  Furthermore, we give several examples of such K\"ahler reductions when $M$ is a Hermitian symmetric space of compact type (see Subsection \ref{subsec3.3}). 

Finally, applying these results, we give several examples of minimal Lagrangian submanifolds in some homogeneous K\"ahler-Einstein manifolds via reductions.

\section{K\"ahler reductions and Lagrangian submanifolds}\label{Sec2}
Let $(M, \omega, J)$ be a complex $n$-dimensional K\"ahler manifold, where $\omega$ is the K\"ahler form and $J$ is the complex structure, and $K$ a real $l$-dimensional connected compact Lie group. We define the compatible Riemannian metric $g$ by $g(\cdot, \cdot):=\omega(\cdot, J\cdot)$. Suppose $K$ acts on $M$ as holomorphic isometries and the action is Hamiltonian. We fix a moment map of the $K$-action by $\mu: M\rightarrow \fk^*$, where $\fk$ is the Lie algebra of $K$. We refer to \cite{Audin} for general facts of the moment map.

Let $L$ be a $K$-invariant connected Lagrangian submanifold in $M$ and $\phi: L\rightarrow M$ the embedding of $L$, i.e., $\phi$ satisfies $\phi^*\omega=0$ and ${\rm dim}_{\mathbb{R}}L=n$. Then, there exists $c\in \mathfrak{z(k^*)}:=\{c\in\fk^*; {\rm Ad}^*(k)c=c\ \forall k\in K\}$ so that $\phi(L)$ is contained in the level set $\mu^{-1}(c)$ of $\mu$ (cf. Lemma 2.3 in \cite{Dong}). Because of this reason, we always assume $c\in\mathfrak{z(k^*)}$ throughout this section. We remark that we can take another moment map $\mu':=\mu-c$ so that $\mu'^{-1}(0)=\mu^{-1}(c)$ since $c\in\mathfrak{z(k^*)}$. However, we shall use the fixed $\mu$ because $c$ depends on each $K$-invariant Lagrangian submanifold.  

Since $c\in\mathfrak{z(k^*)}$, $K$ acts on $\mu^{-1}(c)$. Throughout this section, we assume $c$ is a regular value of the moment map and the action $K\curvearrowright \mu^{-1}(c)$ is free.  Then, $\mu^{-1}(c)$  and the quotient space $M_c:=\mu^{-1}(c)/K$ are smooth manifolds and $M_c$ inherits a natural K\"ahler structure in the sense of Theorem 7.2.3 in \cite{Futaki2}. We denote the inclusion and the natural projection of $\mu^{-1}(c)$ by $\iota: \mu^{-1}(c)\rightarrow M$ and $\pi: \mu^{-1}(c)\rightarrow M_c$, respectively. 

We recall the structure of the K\"ahler quotient space according to Section 7 in \cite{Futaki2}. For $X\in \fk$, we denote the fundamental vector field at $p\in M$ by $\tilde{X}_p:=\frac{d}{dt}|_{t=0}{\rm exp}(tv)\cdot p$, and set $\fk_p:=\{\tilde{X}_p; X\in \fk\}=T_p(K\cdot p)$.  For any $p\in \mu^{-1}(c)$, we have
\begin{align}\label{eq1}
T_p\mu^{-1}(c)=E_p\oplus \fk_p\quad {\rm and}\quad  T_p^{\perp}\mu^{-1}(c)=J\fk_p,
\end{align}
where  $E_p$ is the orthogonal complement of $\fk_p$ in $T_p\mu^{-1}(c)$.  We define vector bundles over $\mu^{-1}(c)$ by
\begin{align*}
E:=\bigcup_{p\in \mu^{-1}(c)}E_p\quad {\rm and}\quad F:=\bigcup_{p\in \mu^{-1}(c)}(\fk_p\oplus J\fk_p).
\end{align*}
  By definition of $\pi$, $\pi_*: E_p\rightarrow T_{\pi(p)}M_c$ is an isomorphism.
We define an almost complex structure $J_c$ and the Riemannian metric $g_c$ on $M_c$ by $(\pi_*)|_E\circ J=J_c\circ (\pi_*)|_E$ and $g|_{E_p}=\pi^*g_c$, respectively. Then, $J_c$ is integrable and $g_c$ defines a K\"ahler metric on $M_c$ (\cite{Futaki2}). Moreover, the K\"ahler form $\omega_c(\cdot, \cdot):=g_c(J_c\cdot, \cdot)$ satisfies $\pi^*\omega_c=\iota^*\omega$.

Notice that the projection $\pi: \mu^{-1}(c)\rightarrow M$ is a Riemannian submersion. Let $\nabla$ and $\nabla_c$ be the Levi-Civita connections of $(M,g)$ and $(M_c,g_c)$, respectively. Then, we have
\begin{align}\label{eq2}
(\nabla_c)_{Z_1}Z_2=\pi_*\{p_1(\nabla_{Z'_1}Z'_2)\}
\end{align}
for $Z_1, Z_2\in \Gamma(TM_c)$, where $Z'_i$ is the unique $K$-invariant section of $E$ so that $\pi_*(Z'_i)=Z_i$ and $p_1: T_pM\rightarrow E_p$ is the orthogonal projection.

Fix a basis $\{\tilde{v}_i\}_{i=1}^l$ of $\fk_p$, where $v_i\in \fk$ for $i=1,\ldots, l$, and define
\begin{align*}
\nu:=\tilde{v}_1^*\wedge\ldots\wedge \tilde{v}_l^*,
\end{align*} 
where $\{\tilde{v}_i^*\}_{i=1}^l$ is the dual basis of $\{\tilde{v}_i\}_{i=1}^l$, i.e., $\tilde{v}_i^*:=g(\tilde{v}_i,\cdot)$. It is easy to verify that $\nu$ and its norm $|\nu|_g$ are $K$-invariant.  Thus, we obtain a well-defined function $|\check{\nu}|$ on $M_c$ so that  $|\nu|_g=|\check\nu|\circ \pi$.
Moreover, the volume of the $K$-orbit $\mathcal{O}_p=K\cdot p$ through $p\in \mu^{-1}(c)$ with respect to the induced metric from $g$ is given by  
\begin{align}\label{eq3}
vol_g(\mathcal{O}_p)=\int_{\mathcal{O}_p} \sqrt {{\rm det}(g(\tilde{v}_i, \tilde{v}_j))}\nu=|\nu|_g(p)\int_{\mathcal{O}_p} \nu.
\end{align}
Since $K$ acts on $\mu^{-1}(c)$ freely, any $K$-orbit contained in $\mu^{-1}(c)$ is a principal orbit in $\mu^{-1}(c)$. In particular, the orbits are diffeomorphic to each other, and $\int_{\mathcal{O}_p} \nu$ is independent of the choice of $p\in \mu^{-1}(c)$.

According to Hsiang-Lawson \cite{HL}, we define the {\it Hsiang-Lawson metric ${g}_{HL}$ of $g$} on the quotient space $M_c$ by
\begin{align*}
{g}_{HL}(x):=vol_g(\mathcal{O}_p)^{2/(n-l)}g_c(x)
\end{align*}
for $x\in M_c$ and $p\in \pi^{-1}(x)$, where $n-l$ is the cohomogeneity of the $K$-invariant Lagrangian embedding $\phi: L\rightarrow M$, i.e., $n-l={\rm dim}_{\mathbb{R}}L-{\dim}_{\mathbb{R}}\mathcal{O}_p$. Notice that $n-l={\rm dim}_{\mathbb{C}}M_c$ in our setting.  Also, we define a globally conformal symplectic form ${\omega}_{HL}$ by ${\omega}_{HL}(\cdot, \cdot):={g}_{HL}(J_c\cdot, \cdot)$.

For the Lagrangian embedding $\phi: L\rightarrow M$, we have the following isomorphisms: 
\begin{align*}
\begin{aligned}
T_pL\ \tilde{\rightarrow}\ T_p^{\perp}L\ \tilde{\rightarrow}\ T_p^*L, \quad V\mapsto JV\mapsto \alpha_V:=\phi^*(i_V\omega)
\end{aligned}
\end{align*} 
for any $p\in L$, where $T_p^{\perp}L$ is the normal space of $T_pL$ in $T_pM$ with respect to $g$ and $i$ denotes the inner product. Moreover, we have a decomposition
\begin{align*}
T_pL={E}^l_p\oplus \mathfrak{k}_p,
\end{align*}  
 where ${E}^l_p$ is the orthogonal complement of $\mathfrak{k}_p$ in $T_pL$.  Because $L$ is Lagrangian and $E_p$ is a complex subspace of $T_pM$, we have $E_p={E}^l_p\oplus J{E}^l_p$ for $p\in L$.   
 
 Since $K$ acts on $L\subset \mu^{-1}(c)$ freely, we obtain a smooth manifold ${L}_c:=\pi(L)$ and the embedding ${\phi}_c: {L}_c \rightarrow M_c$ of $L_c$.  
Then, ${\phi}_c$ is a Lagrangian embedding into $M_c$ since $\pi^*\omega_c=\iota^*\omega$ and  ${\rm dim}_{\mathbb{R}}{L}_c={\rm dim}_{\mathbb{C}}M_c$.   We call $\phi_c$ the {\it reduced} Lagrangian embedding of $\phi$.
 By definition, we have $\pi\circ \phi=\phi_c\circ \pi$ on $L$ and $\pi_*|_{{E}^l_p}: {E}^l_p\ {\rightarrow} \ T_{\pi(p)}L_c$ and $\pi_*|_{J{E}^l_p}: J{E}^l_p\rightarrow J_cT_{\pi(p)}L_c=T^{\perp}_{\pi(p)}L_c$ are isomorphisms.

\subsection{The Ricci form}\label{subsec2.1}
In this subsection, we derive a formula of the Ricci form of the K\"ahler quotient space. This was first computed by Futaki \cite{Futaki}, and  we shall slightly improve his result. We refer to \cite{Futaki} for more details.

Let $M$ and $K$ be as described above. Suppose $c\in \mathfrak{z(k^*)}$ is a regular value of the moment map $\mu$, and $K$ acts freely on $\mu^{-1}(c)$.
Since $E$ and $F$ are $J$-invariant, we have decompositions $E\otimes \mathbb{C}=E^{1,0}\oplus E^{0,1}$ and $F\otimes \mathbb{C}=F^{1,0}\oplus F^{0,1}$ into $\pm{\sqrt{-1}}$-eigendecompositions of $J$. Then, we see $\iota^*T^{1,0}M=E^{1,0}\oplus F^{1,0}$, where $\iota: \mu^{-1}(c)\rightarrow M$. We take a $K$-invariant unitary basis $\{\eta_1,\ldots,\eta_{n-l}\}$ of $E^{1,0}$ with respect to the hermitian metric $h:=g+\sqrt{-1}\omega$. For the real basis $\{\tilde{v}_k\}_{k=1}^{l}$ of $\fk_p$, we set $\xi_k:=(\tilde{v}_k-\sqrt{-1}J\tilde{v}_k)/2\in \Gamma(T^{1,0}M)$ for $k=1,\ldots, l$.  Obviously, $\xi_k$ is $K$-invariant for any $k$. For the immersion $\iota: \mu^{-1}(c)\rightarrow M$,   we define sections of $\Lambda^{l}F^{1,0}$ and $\wedge^n T^{1,0}M$ by 
\begin{align*}
\begin{aligned}
\xi:= \xi_1\wedge\ldots \wedge \xi_l\quad {\rm and}\quad \Omega:={\eta}_1\wedge\ldots \wedge{\eta}_{n-l}\wedge\xi_1\wedge\ldots\xi_l,
\end{aligned}
\end{align*}
respectively.  By definition, these sections are $K$-invariant.

\begin{lemma}\label{lem2.1}
We have $||\xi||_h=|\nu|_g$.
\end{lemma}
\begin{proof}
Since any $K$-orbit contained in a level set $\mu^{-1}(c)$ is isotropic, we have $h(\xi_i,\xi_j)=g(\tilde{v}_i, \tilde{v}_j)+\sqrt{-1}\omega(\tilde{v}_i, \tilde{v}_j)=g(\tilde{v}_i, \tilde{v}_j)$. Thus, $||\xi||^2_h={\rm det}_{\mathbb{C}}(h(\xi_i,\xi_j))_{ij}={\rm det}_{\mathbb{R}}(g(\tilde{v}_i,\tilde{v}_j))_{ij}=|\nu|_g^2$.
\end{proof}

Denote the Ricci forms of $M$ and $M_c$ by $\rho$ and $\rho_c$, respectively (Note that our definition of the Ricci form is different from the ones in \cite{Futaki}. The relationship between the Ricci form $\gamma$ used in \cite{Futaki} and $\rho$ is given by $\gamma=\frac{1}{2\pi}\rho$). Then, the equality (3.3), Lemma 3.4 in \cite{Futaki} and Lemma \ref{lem2.1} implies
\begin{align}\label{eq4}
\pi^*\rho_c=\iota^*\rho+\pi^*dd^c\log |\check \nu|-\sqrt{-1}d\theta_v,
\end{align}
where $dd^c=2\sqrt{-1}\partial\overline{\partial}$ and $\theta_v$ is a 1-form on $\mu^{-1}(c)$ defined by
\begin{align*}
\begin{cases}
     \iota^*\nabla_{\tilde{X}}\Omega=\theta_v(\tilde{X})\Omega,\ {\rm for}\ X\in \fk\quad {\rm and}\\
     \theta_v(Z)=0\ {\rm for}\ Z\in E_p.
    \end{cases}
\end{align*}

In order to compute $\theta_v$, we need the following lemma which holds for any K\"ahler manifold. A proof is similar to Lemma 4.4 in \cite{Futaki}. Thus, we omit the proof.

\begin{lemma}[cf. \cite{Futaki}]\label{lem2.2}
Let  $s$  be a section of $\wedge^n T^{1,0}M$. Then, for any $X\in \fk$, we have
\begin{align*}
\mathcal{L}_{\tilde{X}}s=\nabla_{\tilde{X}}s-\frac{\sqrt{-1}}{2}\Delta\mu^Xs,
\end{align*}
where $\mathcal{L}$ is the Lie derivative and $\mu^X:=\langle\mu, X \rangle$.
\end{lemma}

Define a 1-form ${\gamma}'_c$ on ${\mu}^{-1}(c)$ by
\begin{align}\label{eq5}
    \begin{cases}
     {\gamma}'_c(p)(\tilde{X}_p):=-\frac{1}{2}{\rm div}(J\tilde{X})_p=\frac{1}{2}\Delta \mu^X(p)\ {\rm for}\ X\in \fk\quad {\rm and}\\
      {\gamma}'_c(p)(Z):=0\ {\rm for}\ Z\in E_p.
    \end{cases}
\end{align}
Then,  Lemma \ref{lem2.2} shows $\iota^*(\mathcal{L}_{\tilde{X}}s)=\iota^*(\nabla_{\tilde{X}}s)-{\sqrt{-1}}\gamma'_c(\tilde{X})\iota^*s$.
Taking $s$ as the $K$-invariant section $\Omega$, we have
$\iota^*\nabla_{\tilde{X}}\Omega=\sqrt{-1}{\gamma}'_c(\tilde{X})\Omega$. Therefore, we obtain
\begin{align}\label{eq6}
\theta_v=\sqrt{-1}{\gamma}'_c.
\end{align}

\begin{proposition}\label{prop2.3}
We have
\begin{align}\label{eq7}
\pi^*\rho_c=\iota^*\rho+d{\gamma}'_c+\pi^*dd^c\log|\check{\nu}|. 
\end{align}
Here, the 2-form $d{\gamma}'_c$ satisfies
\begin{eqnarray*}
d{\gamma}'_c(Z,W)(p)=
\left\{
    \begin{array}{l}
      2{\gamma}'_c(p)(JB'(Z, JW)) \quad {\rm for}\ Z,W\in E_p\\
      0\quad {\rm for\  otherwise}
    \end{array}
    \right.
    ,
\end{eqnarray*}
where $B'$ is the second fundamental form of $\iota: \mu^{-1}(c)\rightarrow M$.
\end{proposition}
\begin{proof}
\eqref{eq7} follows from \eqref{eq4} and \eqref{eq6}. If $Z\in E_p$, then ${\gamma}'_c(Z)=0$ by definition. Thus, for $Z,W\in E_p$, we see
\begin{align*}
d{\gamma'_c}(Z,W)=-{\gamma'_c}([Z,W])=-{\gamma'_c}([Z,W]^{\top_{\fk_p}}),
\end{align*}
where ${\top_{\fk_p}}$ means the orthogonal projection onto $\fk_p$.
Here, by (3.6) and (4.7) in \cite{Kobayashi}, we have 
\begin{align*}
[Z,W]^{\top_{\fk_p}}=2(\nabla'_ZW)^{\top_{\fk_p}}=:2C(Z,W)=-2JB'(Z,JW),
\end{align*}
where $\nabla'$ is the Levi-Civita connection on $T\mu^{-1}(c)$.
Therefore, we obtain $d{\gamma'_c}(Z,W)=2{\gamma}'_c(p)(JB'(Z, JW))$ for $Z,W\in E_p$. One can easily check that $d\gamma'_c(X,Y)=0$ for other pairs $X,Y$.
\end{proof}

\subsection{The mean curvature form}\label{subsec2.2}
Let $L$ be the Lagrangian submanifold in $M$ which is contained in $\mu^{-1}(c)$, and $K\cdot p$ the $K$-orbit through $p\in L$. We denote the mean curvature vectors of $L$ and $K\cdot p$ in $\mu^{-1}(c)$ with respect to $\iota^*g$ by $H'$ and $\hat{H}$, respectively.  Also, $H_c$ is denoted by the mean curvature vector of the reduced Lagrangian submanifold $L_c$ in $M_c$ with respect to the K\"ahler metric $g_c$. 

The following formula for $\hat H$ can be proven in a general setting (cf. \cite{Pacini}). However, we give a proof for the convenience of the reader. We denote the Levi-Civita connection of $T\mu^{-1}(c)$ by $\nabla'$. 

\begin{lemma} \label{lem2.4}
We have $\hat{H}_p=-\nabla'\log|\nu|(p)$. In particular, $\pi_*H'=H_c-(\nabla_c\log|\check{\nu}|)^{\perp_c}$, where $\perp_c$ denotes the orthogonal projection onto the normal space of $T_{\pi(p)}L_c$ in $T_{\pi(p)}M_c$.
\end{lemma}
\begin{proof}
We induce $\nabla'$ to the subbundle $\bigcup_{p\in \mu^{-1}(c)} \fk_p$ of $T\mu^{-1}(c)$. We denote the induced connection and its connection form in the trivialization $\nu$ by ${\nabla}^v$ and $\theta^v$, respectively, i.e., $\nabla^v\nu=\theta^v\otimes \nu$. Then, $\theta^v$ splits into $\theta^v=\theta^v_h+\theta^v_v$, where 1-forms $\theta^v_h$ and $\theta^v_v$ are defined by
\begin{align*}
\begin{aligned}
&\theta^v_h(Z)=\theta^v(Z),\quad \theta^v_h(V)=0,\quad \theta^v_v(Z)=0,\quad \theta^v_v(V)=\theta^v(V)
\end{aligned}
\end{align*}
for $Z\in {E}_p$ and $V\in \fk_p$. We shall calculate $\theta_h^v$ in different two ways. Take a $K$-invariant local basis $\{Z_1,\ldots, Z_{2n-2l}\}$ of $E$. Then we have $[Z_i, \tilde{v}_j]=0$ and $Z_i\perp \tilde{v}_j$ for any $i=1,\ldots, 2n-2l$ and $j=1,\ldots, l$. Therefore,  we see
\begin{align*}
\begin{aligned}
\nabla_{{Z_i}}^v \nu&=\sum_{j=1}^l\tilde{v}_1\wedge\ldots\wedge (\nabla_{Z_i}^v\tilde{v}_j)\wedge\ldots\wedge \tilde{v}_l\\
&=\sum_{j=1}^l\tilde{v}_1\wedge\ldots\wedge \Big{(}\sum_{k,m=1}^lg^{km}g(\nabla_{Z_i}'\tilde{v}_j, \tilde{v}_k)\tilde{v}_m\Big{)}\wedge\ldots\wedge \tilde{v}_l\\
&=\Big{(}\sum_{k,j=1}^lg^{kj}g(\nabla_{Z_i}'\tilde{v}_j, \tilde{v}_k)\Big{)}\nu=\Big{(}\sum_{j,k=1}^lg^{jk}g(\nabla_{\tilde{v}_j}'{Z_i}, \tilde{v}_k)\Big{)}\nu\\
&=-\Big{(}\sum_{j,k=1}^lg^{jk}g({Z_i}, \nabla_{\tilde{v}_j}'\tilde{v}_k)\Big{)}\nu=-g({Z_i}, \hat{H})\nu
\end{aligned}
\end{align*}
for any $Z_i$, where $g_{jk}:=g(\tilde{v}_j, \tilde{v}_k)$ and $(g^{jk})_{j,k=1,\ldots, l}$ denotes the inverse matrix of $(g_{jk})_{j,k=1,\ldots, l}$. 
On the other hand, since $\nabla^v$ is the metric connection, we see
\begin{align*}
\nabla_{Z_i}^v \nu= g\Big{(}\nabla_{Z_i}^v\nu, \frac{\nu}{|\nu|}\Big{)}\frac{\nu}{|\nu|}=(d\log |\nu|)({Z_i})\nu.
\end{align*}
Therefore, we obtain $\theta_h^v=-g(\hat{H},\cdot)|_{E}=(d\log|\nu|)|_{E}$. Moreover, we have $-g(\hat{H},\cdot)|_{\fk_p}=(d\log|\nu|)|_{\fk_p}=0$, and hence, $-g(\hat{H},\cdot)|_{T_p\mu^{-1}(c)}=d\log|\nu|$. This implies $\hat{H}_p=-\nabla'\log|\nu|(p)$.

Take an orthonormal basis $\{e_i\}_{i=1}^{n-l}$ of $L_c$. Then, we have an orthonormal basis $\{e_i'\}_{i=1}^{n-l}$ of ${E}_p$ satisfying $\pi_*e_i'=e_i$ for $i=1,\ldots, n-l$. Also, we take  an orthonormal basis $\{\nu_j\}_{j=1}^{l}$ of $\fk_p$. Then, we see 
\begin{align*}
\begin{aligned}
\pi_*H'&=\pi_*\Big{(}\sum_{i=1}^{n-l}(\nabla'_{e_i'}e_i')^{\perp'}+\sum_{j=1}^l(\nabla'_{\nu_j}\nu_j)^{\perp'}\Big{)}=\sum_{i=1}^{n-l}(\pi_*\nabla'_{e_i'}e_i')^{\perp_c}+(\pi_*\hat{H})^{\perp_c}\\
&=\sum_{i=1}^{n-l}\{(\nabla_{c})_{e_i}e_i\}^{\perp_c}-(\nabla_c\log|\check{\nu}|)^{\perp_c}=H_c-(\nabla_c\log|\check{\nu}|)^{\perp_c},
\end{aligned}
\end{align*}
where $\perp'$ denotes the orthogonal projection onto the normal space of $T_pL$ in $T_p\mu^{-1}(c)$.
\end{proof}

Denote the mean curvature vectors of $\phi: L\rightarrow M$ with respect to $g$  and $\phi_c:L_c\rightarrow M_c$ with respect to the Hsiang-Lawson metric $g_{HL}$ by $H$ and $H_{HL}$, respectively. Then, we define the mean curvature forms by  $\alpha_H:=\phi^*(i_H\omega)$, $\beta_{H_c}:=\phi^*_c(i_{H_c}\omega_c)$ and $\beta'_{H_{HL}}:=\phi_c^*(i_{H_{HL}}\omega_{HL})$.

\begin{proposition}\label{prop2.5}
Let $M$ be a K\"ahler manifold, $K$ a connected compact Lie group acting on $M$ as holomorphic isometries and a Hamiltonian way, and $\phi: L\rightarrow M$ a $K$-invariant Lagrangian embedding. Suppose $\phi(L)\subset {\mu}^{-1}(c)$ for some regular value $c\in\mathfrak{z(k^*)}$ and $K$ acts on $\mu^{-1}(c)$ freely. Then, we have
\begin{align}\label{eq8}
\pi^*{\beta}'_{{H}_{HL}}=\alpha_{H'},
\end{align}
where $H'$ is the mean curvature vector of $L$ in ${\mu}^{-1}(c)$, or equivalently,
\begin{align}\label{eq9}
\begin{aligned}
\pi^*\beta_{H_c}&=\alpha_H+\gamma_c+\pi^*\circ\phi_c^*(d^c\log|\check{\nu}|),
\end{aligned}
\end{align}
where $\gamma_c:=\phi^*{\gamma}'_c$ (see \eqref{eq5} for the definition of $\gamma_c'$).
 \end{proposition}
 
 \begin{proof}
 Under the conformal change ${g}_{HL}=e^{2f}g_c$ for $f\in C^{\infty}(M_c)$, we have 
${H}_{HL}=e^{-2f}\{H_c-(n-l)(\nabla_cf)^{\perp_c}\}$,
where $n-l={\rm dim}_{\mathbb{R}}L_c$. Putting $e^{2f}=(vol_g(\mathcal{O}_p))^{2/(n-l)}$,  we see $f=\frac{1}{n-l}\log |\check{\nu}|+const.$ by \eqref{eq3}. Thus, by Lemma \ref{lem2.4}, we have
\begin{align*}
\begin{aligned}
\pi^*{\beta}'_{{H}_{HL}}&=\pi^*\circ\phi^*_c\{i_{{H}_{HL}}{\omega}_{HL}\}=\phi^*\circ \pi^*\{i_{\{H_c-(\nabla_c \log |\check{\nu}|)^{\perp_c}\}}\omega_c\}\\
&=\phi^*\{\omega_c(\pi_*H', \pi_*\cdot)\}=\phi^*\{i_{H'}\pi^*\omega_c\}=\phi^*\{i_{H'}\iota^*\omega\}=\alpha_{H'}.
\end{aligned}
\end{align*}
This proves \eqref{eq8}. Moreover, we see
\begin{align}\label{eq10}
{\beta}'_{{H}_{HL}}=\beta_{H_c}-\phi^*_c(d^c\log|\check{\nu}|).
\end{align}

Finally, we shall show 
\begin{align}\label{eq11}
\alpha_{H'}=\alpha_H+\gamma_c.
\end{align}
Take a local orthonormal frame $\{e_1,\ldots, e_n\}$ of $L$. Since $L$ is Lagrangian, we see
\begin{align*}
\begin{aligned}
\alpha_H(\tilde{X})&=-\sum_{i=1}^n g(\nabla_{e_i}e_i, J\tilde{X})=-\sum_{i=1}^n\{\nabla_{e_i}g(e_i, J\tilde{X})-g(e_i, \nabla_{e_i}J\tilde{X})\}\\
&=\frac{1}{2}{\rm div}J\tilde{X}=-\gamma_{c}(\tilde{X})
\end{aligned}
\end{align*}
for any $X\in \fk$ (A similar calculation is found in \cite{BG}), namely, $\alpha_H|_{\fk_p}=-\gamma_c$. On the other hand, one can easily verify that $\alpha_H|_{{E}^l_p}=\alpha_{H'}$. Thus, we have $\alpha_H=\alpha_{H'}-\gamma_c$.  Substituting \eqref{eq10} and \eqref{eq11} to \eqref{eq8}, we obtain \eqref{eq9}.
 \end{proof}

\eqref{eq8} shows that $H'=0$ if and only if $H_{HL}=0$ since $\pi$ is surjective. This is a special case of the classical fact due to Hsiang-Lawson \cite{HL}. However, our approach is different from theirs.

\begin{remark}\label{rmk2.6}
{\rm For a Lagrangian submanifold in a K\"ahler manifold, we have Dazord's formula: $d\alpha_H=\phi^*\rho$. By taking the exterior derivative of \eqref{eq9}, we obtain  $\pi^*d\beta_{H_c}=d\alpha_H+d\gamma_c+\phi^*\circ\pi^*dd^c\log|\check{\nu}|$.  This coincides with the pull-pack of the formula \eqref{eq7}.}
\end{remark}

\subsection{Minimal Lagrangian submanifolds}\label{subsec2.3}
In this subsection, we suppose furthermore  the Ricci form $\rho$ of the K\"ahler manifold $(M,\omega, J)$ satisfies
\begin{align}\label{eq12}
\rho=C\omega+ndd^cf
\end{align}
for non-zero constant $C$ and $f\in C^{\infty}(M)$, where $n={\rm dim}_{\mathbb{C}}M$. For example, any Fano manifold, i.e., a closed complex manifold with positive first Chern class $c_1(M)$ endowed with a K\"ahler form $\omega'$ so that $\omega'\in 2\pi c_1(M)$ satisfies $\rho=\omega'+ndd^c f$ for a real function $f\in C^{\infty}(M)$. Note that one can rescale the K\"ahler form so that $\omega=C\omega'$ for any positive constant $C$. Then, $\omega$ satisfies the relation \eqref{eq12}.  We remark that a similar condition for the Ricci curvature has been considered in \cite{Behrndt} and \cite{SmoWang}.

If $M$ satisfies \eqref{eq12}, it is somewhat reasonable to consider a conformal change of the metric. Namely, we define a {\it canonical} conformal change of $g$ by $\tilde{g}:=e^{2f}g$.   Also, we define $\tilde{\omega}:=e^{2f}\omega$. Then, $(\tilde{g}, \tilde{\omega}, J)$ defines a {\it globally conformal K\"ahler} structure on $M$. By definition, $\phi^*\omega=0$ if and only if $\phi^*\tilde{\omega}=0$ for an immersion $\phi$ into $M$. Thus, the notion of {Lagrangian} submanifolds is conformal invariant.

Since $K$ acts on $M$ as holomorphic isometries, $\rho$ and $\omega$ are $K$-invariant, and hence, $f$ is a $K$-invariant function. Thus, we obtain a well-defined function $\check{f}\in C^{\infty}(M_c)$ so that  $f=\pi^*\check{f}$. Then, we define a globally conformal K\"ahler metric on the quotient space $M_c$ by $\tilde{g}_c:=e^{2\check{f}}g_c$.  It is obvious that $\pi^*\tilde{g}_c=\iota^*\tilde{g}$. Moreover, we define the Hsiang-Lawson metric $\tilde{g}_{HL}$ on $M_c$ of $\tilde{g}$ by
\begin{align*}
\tilde{g}_{HL}(x):=vol_{\tilde{g}}(\mathcal{O}_p)^{2/(n-l)}\tilde{g}_c(x),
\end{align*}
for $x\in M_c$ and $p\in \pi^{-1}(x)$, where $vol_{\tilde{g}}(\mathcal{O}_p)$ is the volume of $\mathcal{O}_p$ with respect to the globally conformal K\"ahler metric $\tilde{g}$. Since $vol_{\tilde{g}}(\mathcal{O}_p)=e^{lf(p)}vol_{{g}}(\mathcal{O}_p)$, where $l={\rm dim}_{\mathbb{R}}\mathcal{O}_p$, we see
\begin{align}\label{eq13}
\begin{aligned}
\tilde{g}_{HL}(x)&%e^{\{2d/(n-d)\}\check{f}(x)}vol_{{g}}(\mathcal{O}_p)^{2/(n-d)}e^{2\check{f}}g_c
=e^{2f_c(x)}g_c,\ {\rm where}\\
 f_c(x)&:=\log vol_{{g}}(\mathcal{O}_p)^{1/(n-l)}+\frac{n}{n-l}\check{f}(x).
\end{aligned}
\end{align}

First, we mention the moment map of an action of holomorphic isometries.  The following proposition can be observed by the result in \cite{Futaki} (Our formulation is inspired by \cite{Podesta}).

\begin{proposition}\label{prop2.7}
Let $M$ be a K\"ahler manifold satisfying $\rho=C\omega+ndd^c f$ with $C\neq 0$ and $K$ a connected compact Lie group. Assume $K$ acts on $M$ preserving $(\omega, g, J)$.  Then, the action is Hamiltonian, and a moment map $\tilde{\mu}: M\rightarrow \fk^*$ is given by
\begin{align}\label{eq14}
\langle \tilde{\mu}(p), X\rangle=\frac{1}{C}\Big{\{}-\frac{1}{2}{\rm div}J\tilde{X}_p+nd^cf_p(\tilde{X}_p)\Big{\}},
\end{align}
for $X\in \fk$, where ${\rm div}$ is the divergence operator on $M$ with respect to $g$. We call $\tilde{\mu}$ the {\rm canonical moment map} for the $K$-action.
\end{proposition}

\begin{proof}
Since $K$ acts on $M$ as holomorphic isometries, it is easy to verify that  $\tilde{\mu}$ is $K$-equivariant, namely, $\tilde{\mu}(kp)={\rm Ad}^*(k^{-1})\tilde{\mu}(p)$ for any $k\in K$. Moreover, the fundamental vector field $\tilde{X}$ for $X\in \fk$ is infinitesimal automorphic, namely, $\mathcal{L}_{\tilde{X}}J=0$, or equivalently, $[\tilde{X}, JY]=J[\tilde{X}, Y]$ for any $Y\in \Gamma(TM)$. This implies 
\begin{align}\label{eq15}
\nabla_{JY}\tilde{X}=J\nabla_{Y}\tilde{X}
\end{align}
since $\nabla J=0$.

By Proposition 1.1 in \cite{Podesta}, we have
\begin{align}\label{eq16}
Y\Big{(}-\frac{1}{2}{\rm div}J\tilde{X}\Big{)}=\rho(\tilde{X}, Y)
\end{align}
for any $Y\in T_pM$ (Note that the signs of of $\omega$ and $\rho$ in \cite{Podesta} are different from ours).  On the other hand, we have
\begin{align}\label{eq17}
\begin{aligned}
Y\{d^cf(\tilde{X})\}&=-Y\{df(J\tilde{X})\}=-Y\{g(\nabla f, J\tilde{X})\}\\
&=-{\rm Hess}_f(Y, J\tilde{X})-g(\nabla f, \nabla_YJ\tilde{X}).
\end{aligned}
\end{align}
Since $f$ is a $K$-invariant function, we have $\tilde{X}f=0$.   By using this fact and \eqref{eq15}, we compute
\begin{align*}
\begin{aligned}
g(\nabla f, \nabla_YJ\tilde{X})&=g(\nabla f, J\nabla_Y\tilde{X})=g(\nabla f, \nabla_{JY}\tilde{X})=JY(\tilde{X}f)-g(\nabla_{JY}\nabla f, \tilde{X})\\
&=-{\rm Hess}_f(JY, \tilde{X})=-{\rm Hess}_f(\tilde{X}, JY)
\end{aligned}
\end{align*}
Thus, a straight forward calculation shows that \eqref{eq17} becomes
\begin{align}\label{eq18}
Y\{d^cf(\tilde{X})\}=-{\rm Hess}_f(Y, J\tilde{X})+{\rm Hess}_f(\tilde{X},JY)=-dd^cf(\tilde{X}, Y).
\end{align}
Therefore, by \eqref{eq16}, \eqref{eq18} and the assumption of $\rho$, we obtain
\begin{align*}
\langle Y\tilde\mu_p, X\rangle=\frac{1}{C}\{\rho(\tilde{X},Y)-ndd^cf(\tilde{X},Y)\}=\omega(\tilde{X}, Y).\end{align*}
This proves $d\tilde\mu^X=i_{\tilde{X}}\omega$ for any $X\in \fk$. Thus, $\tilde{\mu}$ is a moment map.
\end{proof}

\begin{remark}\label{rmk2.8}
{\rm If we define a weighted Laplacian on $M$ by $\Delta_fu:=\Delta u-2ng(du, df)$ for $u\in C^{\infty}(L)$, then we see from \eqref{eq16} that $\Delta_f\tilde{\mu}^X=2C\tilde{\mu}^X$ for any $X\in \fk$ since $J\tilde{X}=\nabla \tilde\mu^X$, namely, $\tilde{\mu}^X$ is an eigenfuntion of $\Delta_f$. In particular, if $M$ is closed, then the canonical moment map is characterized by $\int_{M}\tilde{\mu}^X\tilde{\omega}^n=0$. See \cite{Futaki} for more details of the canonical moment map when $M$ is a Fano manifold. }
\end{remark}

Replacing $\mu$  by  the canonical moment map $\tilde{\mu}$, we have from \eqref{eq5}
\begin{align}\label{eq19}
    \begin{cases}
     {\gamma}'_c(\tilde{X}_p)=Cc(X)-nd^cf_p(\tilde{X}_p)\ {\rm for}\ X\in \fk\quad {\rm and}\\
      {\gamma}'_c(p)(Z)=0\ {\rm for}\ Z\in E_p.
    \end{cases}
\end{align}
 for any $X\in \fk$ and $p\in \tilde{\mu}^{-1}(c)$.  If $c=0$, then \eqref{eq19}  implies
 \begin{align}\label{eq20}
{\gamma}'_0=-n\{\iota^*d^cf-\pi^*d^c\check{f}\}.
 \end{align}
Thus, the Ricci form $\rho_0$ of the K\"ahler quotient space $M_0=\tilde{\mu}^{-1}(0)/K$ satisfies
\begin{align*}
\begin{aligned}
\pi^*\rho_0&=\iota^*\rho-n\{\iota^*dd^cf-\pi^*dd^c\check{f}\}+\pi^*dd^c\log|\check{\nu}|\\
&=\pi^*\{C\omega_0+dd^c(\log|\check\nu|+n\check{f})\}
\end{aligned}
\end{align*}
by \eqref{eq7} and \eqref{eq20}.  Because $\pi$ is surjective, this shows that 
 \begin{align}\label{eq21}
\begin{aligned}
\rho_0&=C\omega_0+dd^c(\log|\check{\nu}|+n\check{f})=C\omega_0+(n-l)dd^cf_0,
\end{aligned}
\end{align}
where $f_0$ is given by \eqref{eq13}. Namely, we have the following which slightly generalizes the result of Futaki \cite{Futaki}:

\begin{proposition} \label{prop2.9}
Let $M$ be a K\"ahler manifold satisfying $\rho=C\omega+ndd^cf$ with $C\neq 0$ and $K\curvearrowright M$ an action of  holomorphic isometries.  If $0$ is a regular value of the canonical moment map $\tilde{\mu}$ of the $K$-action and $K$ acts on $\tilde{\mu}^{-1}(0)$ freely, then the Ricci form of the K\"ahler quotient space $(M_0=\tilde{\mu}^{-1}(0)/K, \omega_0, J_0)$ satisfies \eqref{eq21}.

In particular, the Hsiang-Lawson metric $\tilde{g}_{HL}$ on $M_0$ of $\tilde{g}$ coincides with the canonical conformal change of ${g}_0$, i.e., $\tilde{g}_{HL}=e^{2f_0}g_0$.
\end{proposition}

\begin{remark}\label{rmk2.10}
{\rm (i) When $M$ is K\"ahler-Einstein, then so is $(M_0, \omega_0, J_0)$ if and only if $|\nu|$ is constant on $\tilde{\mu}^{-1}(0)$, or equivalently,  every $K$-orbits contained in $\tilde{\mu}^{-1}(0)$ has the same volume (cf. Corollary 3 in \cite{Futaki}). 

(ii) If $c\neq 0$, then  $d\gamma'_c$ depends on the second fundamental form of $\iota: \tilde{\mu}^{-1}(c)\rightarrow M$. See Proposition \ref{prop2.3}}.
\end{remark}

Let $\phi: L\rightarrow M$ be a Lagrangian immersion. Denote the mean curvature vector of $\phi$ with respect to $\tilde{g}=e^{2f}g$  by $\tilde{H}$, and set $\tilde{\alpha}_{\tilde{H}}:=\phi^*(i_{\tilde{H}}\tilde\omega)$. By definition, we have
\begin{align}\label{eq22}
\begin{aligned}
\tilde{\alpha}_{\tilde{H}}&=\alpha_H-n\phi^*d^cf, 
%\tilde{\beta}_{\tilde{H}_{HL}}&=\beta_{H_c}-(n-l)\phi^*_cd^cf_c=\beta_{H_c}-\phi^*_0d^c(\log|\check{\nu}|+n\check{f}).
\end{aligned}
\end{align}
We remark that $\tilde{\alpha}_{\tilde{H}}$ is a  closed 1-form. In fact, Dazord's formula $d\alpha_H=\phi^*\rho$ implies
\begin{align*}
d\tilde{\alpha}_{\tilde{H}}=\phi^*\rho-n\phi^*dd^cf=\phi^*(C\omega)=0.
\end{align*}
See Remark \ref{rmk2.13} in below on this point.  Also, we denote the mean curvature vector of  $\phi_0:L_0\rightarrow M_0$ with respect to $\tilde{g}_{HL}$ by $\tilde{H}_{HL}$, and we set 
\begin{align*}
\tilde{\beta}_{\tilde{H}_{HL}}:=\phi_0^*(i_{\tilde{H}_{HL}}\tilde{\omega}_{HL}),
\end{align*}
 where $\tilde{\omega}_{HL}(\cdot, \cdot):=\tilde{g}_{HL}(J_0\cdot, \cdot)$. A similar argument shows that $\tilde{\beta}_{\tilde{H}_{HL}}$ is also a closed form.

Now, we state the first main result of the present paper:

\begin{theorem}\label{thm2.11}
Let $(M, \omega, J)$ be a complex $n$-dimensional K\"ahler manifold, $K$ a connected compact Lie group acting on $M$ as holomorphic isometries and  $\phi: L\rightarrow M$ a $K$-invariant Lagrangian embedding of a manifold $L$. Suppose the Ricci form of $M$ satisfies $\rho=C\omega+ndd^c f$ for some constant $C\neq 0$ and $f\in C^{\infty}(M)$.  Moreover, we define a globally conformal K\"ahler metric $\tilde{g}$ on $M$ by $\tilde{g}:=e^{2f}g$. Then, we have the following:
\begin{description}
\item[\rm (a)] If $\phi$ is minimal with respect to $\tilde{g}$, i.e., $\tilde{\alpha}_{\tilde{H}}=0$, or more generally $\tilde{\alpha}_{\tilde{H}}$ is exact, then $\phi(L)$ is contained in the $0$-level set of the canonical moment map $\tilde{\mu}$ of the $K$-action.
\item[\rm (b)] Suppose $0\in \fk^*$ is a regular value of $\tilde{\mu}$ and $K$ acts on $\tilde{\mu}^{-1}(0)$ freely. Furthermore, we assume $\phi(L)$ is contained in $\tilde{\mu}^{-1}(0)$. Then,  we have
\begin{align}\label{eq23}
\pi^*\tilde{\beta}_{\tilde{H}_{HL}}=\tilde{\alpha}_{\tilde{H}}.
\end{align}
In particular, $\phi$ is minimal with respect to $\tilde{g}$ if and only if so is the reduced Lagrangian embedding $\phi_0: L_0\rightarrow M_0$ with respect to the Hsiang-Lawson metric of $\tilde{g}$. 
\end{description}
\end{theorem}

\begin{proof}
  We assume $\phi(L)$ is contained in $\tilde{\mu}^{-1}(c)$ for some $c\in \mathfrak{z(k^*)}$. Suppose $\tilde{\alpha}_{\tilde{H}}$ is an exact form,  i.e., there exists a smooth function $u\in C^{\infty}(L)$ so that  $\tilde{\alpha}_{\tilde{H}}=du$.  Then, \eqref{eq11} and \eqref{eq22} implies
  \begin{align}\label{eq24}
  du=\alpha_{H'}-\gamma_c-n\phi^*d^cf.
  \end{align}  
  Since $\tilde{g}$ and $L$ are $K$-invariant, so is the 1-form $\tilde{\alpha}_{\tilde{H}}$, and hence, $u$ is a $K$-invariant function.  Moreover, by \eqref{eq19}, we have
  \begin{align*}
\begin{aligned}
\gamma_c(\tilde{X})+n(\phi^*d^cf)(\tilde{X})=Cc(X)
\end{aligned}
\end{align*}
  for any $X\in \fk_p$. Thus, substituting $\tilde{X}$ to (24), we obtain
  \begin{align*}
  0=\alpha_{H'}(\tilde{X})-Cc(X)=g(JH', \tilde{X})-Cc(X)=-Cc(X)
  \end{align*}
  since $H'\in E$ and $E$ is $J$-invariant.  Therefore, $c=0$. This proves (a).

  Suppose $\phi(L)$ is contained in $\tilde{\mu}^{-1}(0)$.  Then, we have
  \begin{align}\label{eq25}
\begin{aligned}
\tilde{\beta}_{\tilde{H}_{HL}}&=\beta_{H_0}-(n-l)\phi^*_0d^cf_0=\beta_{H_0}-\phi^*_0d^c(\log|\check{\nu}|+n\check{f}).
\end{aligned}
\end{align}
%Note that $\tilde{\beta}_{\tilde{H}_{HL}}$ is a closed form by a similar argument for $\tilde{\alpha}_{\tilde{H}}$.

By \eqref{eq22} and \eqref{eq25}, \eqref{eq9} becomes
\begin{align}\label{eq26}
\pi^*(\tilde{\beta}_{\tilde{H}_{HL}}+n\phi^*_0d^c\check{f})=\tilde{\alpha}_{\tilde{H}}+\gamma_0+n\phi^*d^cf.
\end{align}
On the other hand,  we have from \eqref{eq20} that $\gamma_0=-n(\phi^*d^cf-\pi^*\circ\phi_0^*d^c\check{f})$. Therefore, \eqref{eq26} becomes
\begin{align*}
\pi^*\tilde{\beta}_{\tilde{H}_{HL}}=\tilde{\alpha}_{\tilde{H}}.
\end{align*}
%Thus, $\tilde{H}=0$ if and only if $\tilde{H}_{{HL}}=0$. 
This proves (b).
\end{proof}

\begin{remark}\label{rmk2.12}
{\rm (i) Theorem \ref{thm2.11} holds even when $M$ is non-compact. If $M$ is a non Ricci flat K\"ahler-Einstein manifold, then $\tilde{g}$ coincides with the K\"ahler metric $g$ (up to constant multiples). However, the quotient structure $(M_0, \omega_0, J_0)$ is not necessarily K\"ahler-Einstein (See Remark \ref{rmk2.10}).

(ii) Theorem \ref{thm2.11} (a) is  a generalization of Proposition 3.2 in \cite{BG} or Theorem 6 in \cite{Pacini} in which  a similar statement  was proved for {\it homogeneous} Lagrangian submanifolds in a K\"ahler-Einstein manifold.  On the other hand, Theorem \ref{thm2.11} (b) is {\it not} a direct consequence of  Hsiang-Lawson's result in \cite{HL}. In fact, a similar statement does not hold in general without the assumption of the Ricci curvature. For example, the Hopf fibration $\pi: S^{2n-1}\rightarrow \mathbb{C}P^{n-1}$ is a typical example of the K\"ahler reduction by the standard $S^1$-action on the complex Euclidean space $M=\mathbb{C}^n$.  In this case,  $M$ is Ricci-flat and one can choose a moment map $\mu$ of the $S^1$-action so that $\mu^{-1}(0)=S^{2n-1}$. However, the moment map is not canonical in the sense of Proposition \ref{prop2.7}. Note that every $S^1$-orbits has the same volume in $S^{2n-1}$, and hence, the Hsiang-Lawson metric is a constant multiple of the Fubini-Study metric on $\mathbb{C}P^{n-1}$. Moreover, if $L_0$ is a compact minimal Lagrangian submanifold in $\mathbb{C}P^{n-1}$, then the pre-image $L=\pi^{-1}(L_0)$ is Hamiltonian-minimal  in $\mathbb{C}^n$. However, $L$ is never minimal in $\mathbb{C}^n$ by the compactness of $L$, and hence,  the minimality of $L_0$ does not correspond to the minimality of $L$ (in the classical sense). }
\end{remark}

\begin{remark}\label{rmk2.13}
{\rm When $(M, \omega, J)$ is not K\"ahler-Einstein, it is necessary to take a conformal change of the original K\"ahler metric $g$ in order to obtain the simple formula (23). In fact, a similar formula between the mean curvature form $\alpha_H$ of $\phi$ with respect to the original K\"ahler metric $g$ and $\beta'_{H_{HL}}$ of $\phi_0$ with respect to the Hsiang-Lawson metric $g_{HL}$ of $g$ is more complicated than \eqref{eq23} (see \eqref{eq9}). 

We remark that the closed 1-forms $\tilde{\alpha}_{\tilde{H}}$ and $\tilde{\beta}_{\tilde{H}_{HL}}$ are referred as {\it generalized mean curvature forms} in \cite{Behrndt} and \cite{SmoWang} or {\it Maslov forms} in \cite{LP} of $\phi$ and $\phi_0$, respectively.  Namely, $\tilde{\alpha}_{\tilde{H}}$ (or $\tilde{\beta}_{\tilde{H}_{HL}}$) is regarded as a "connection form" of a unitary connection $\hat\nabla$ on the trivial bundle $\phi^*K_M$ in the trivialization $\Omega_L$ defined by a unique extension of the volume form of $\phi$: 
\begin{align*}
\hat\nabla \Omega_L=\sqrt{-1}\tilde{\alpha}_{\tilde{H}}\Omega_L,
\end{align*} 
where $\hat\nabla:=\nabla+d^cf\otimes J$ on $TM$ (see Example 2 in \cite{SmoWang} or \cite{KK}), and we use the same symbol for the induced connection on $\phi^*K_M$.  This can be easily shown by using Proposition 4.2 in \cite{LP}. As shown in \cite{Behrndt}, \cite{KK}, \cite{LP} and \cite{SmoWang}, there are several advantages to consider the (closed) Maslov forms in the non K\"ahler-Einstein setting. This point is a crucial difference between $\tilde{\alpha}_{\tilde{H}}$ and ${\alpha}_H$ in our setting, and gives a reason why we consider $\tilde{\alpha}_{\tilde{H}}$ instead of $\alpha_{H}$. From this point of view, \eqref{eq23} shows that  the closed Maslov forms are preserved by $\pi$ whenever $\phi(L)$ is contained in the 0-level set of $\tilde{\mu}$.  
}
\end{remark}

\section{Examples: Reductions of homogeneous hypersurfaces in a K\"ahler manifold}\label{Sec3}
In this section, we give some explicit examples of K\"ahler reductions by using a circle action obtained from a cohomogeneity one action on a K\"ahler manifold. The main result of this section is Theorem \ref{thm3.7}. 

\subsection{Preliminaries}\label{subsec3.1}
Let $(M, \omega)$ be a symplectic manifold.  
Suppose a connected Lie group $G$ acts on $M$ in a Hamiltonian way with the moment map $\mu: M\rightarrow \fg^*$. For any closed subgroup $G'$ of $G$, the induced action $G'\curvearrowright M$ is also a Hamiltonian action. In fact, a moment map is given by $\mu_{G'}:={\rm pr}_{(\fg')^*}\circ \mu$, where $\fg'$ is the Lie algebra of $G'$ and  ${\rm pr}_{\mathfrak{(\fg')}^*}$ is the natural projection onto $(\fg')^*$ (see \cite{Audin}).

  Let $S$ be a connected closed subgroup of $G$ with Lie algebra $\fs$.  Denote the centralizer of $S$ in $G$ by $Z_G(S):=\{g\in G; gs=sg\ \forall s\in S\}$. Then, $Z_G(S)$ is a Lie subgroup of $G$ and the Lie subalgebra is given by $\mathfrak{z_g(s)}:=\{X\in \fg; [X, V]=0\ \forall V\in \fs\}$. If $G=S$, then $Z_S(S)$ is the center of $S$, and we denote it and its Lie algebra by $C(S)$ and $\mathfrak{c(s)}$, respectively.    %中心化はLie部分群?? そうだとするとリー環がそう与えられるのは正しいらしい. 

\begin{lemma}\label{lem3.1}
Let $Z'$ be a connected closed subgroup of $Z_G(S)$ with Lie algebra $\mathfrak{z}'$ and $\mu_{Z'}: M\rightarrow (\mathfrak{z}')^*$ the moment map of the $Z'$-action.  Then,  the level set $\mu^{-1}_{Z'}(c)$ is a $S$-invariant subset in $M$ for any $c\in (\mathfrak{z'})^*$. Conversely, if any level set of the moment map $\mu_{Z'}$ of a connected closed subgroup $Z'$ in $G$ is $S$-invariant and $Z'\curvearrowright M$ is effective, then $Z'$ is a subgroup of $Z_G(S)$. 
\end{lemma}

\begin{proof}
Take a point $p\in \mu^{-1}_{Z'}(c)$. Since ${\rm Ad}(S)X=X$ for any $X\in\mathfrak{z}'\subset \mathfrak{z_g(s)}$, we see
\begin{align*}
\begin{aligned}
\langle \mu_{Z'}(s\cdot p), X\rangle&=\langle {\rm pr}_{(\mathfrak{z'})^*}\circ\mu(s\cdot p), X\rangle=\langle \mu(s\cdot p), X\rangle\\
&=\langle {\rm Ad}^*(s)\mu(p), X\rangle=\langle \mu(p), {\rm Ad}(s^{-1})X\rangle\\
&=\langle \mu(p), X\rangle=\langle \mu_{Z'}(p), X\rangle
\end{aligned}
\end{align*}
for any $s\in S$. This shows $S\cdot p \subseteq \mu^{-1}_{Z'}(c)$ for any $p\in \mu_{Z'}^{-1}(c)$, i.e., $\mu_{Z'}^{-1}(c)$ is $S$-invariant.

Conversely, we assume $Z'$ is a closed connected subgroup of $G$ with Lie algebra $\mathfrak{z}'$ and each level set of $\mu_{Z'}$ is $S$-invariant.  For $X'\in \mathfrak{z}'$ and $V\in \fs$, we define a smooth function on $M$ by $f_{X',V}(p):=\omega_p(\tilde{X'}, \tilde{V})$. Since $\tilde{X'}$ and $\tilde{V}$ are symplectic vector fields on $M$, it turns out that $f_{X',V}$ is a Hamiltonian function with respect to the vector field $[\tilde{V}, \tilde{X'}]$, that is, $df_{X',V}=i_{[\tilde{V}, \tilde{X'}]}\omega$.

   By the assumption, ${\mu}^{X'}_{Z'}$ is a $S$-invariant function, and hence,  we see
    \begin{align*}
  f_{X',V}(p)=\omega_{p}(\tilde{X'},\tilde{V})=d\mu^{X'}_{Z'}(\tilde{V})_{p}=0.
  \end{align*}
  Thus, the smooth function  $f_{X',V}$ is identically zero on $M$. In particular, we see $0\equiv df_{X',V}=i_{[\tilde{V},\tilde{X'}]}\omega$, and hence, $[\tilde{X'}, \tilde{V}]=0$ on $M$ since $\omega$ is non-degenerate. Because $Z'\curvearrowright M$ is effective, this implies  $[X', V]=0$. Therefore, $X'\in \mathfrak{z_g(s)}$ for any $X'\in \mathfrak{z}'$, and hence, $ \mathfrak{z}'$ is a Lie subalgebra of $\mathfrak{z_g(s)}$. Thus, there exists a unique connected Lie subgroup $Z''$ in $Z_G(S)$ with Lie algebra $\mathfrak{z}'$. Since $Z'$ and $Z''$ are connected Lie subgroups of $G$ with the same Lie algebra, we conclude $Z'=Z''\subset Z_G(S)$.
\end{proof}

\begin{remark}\label{rmk3.2}
  {\rm Lemma \ref{lem3.1} is a generalization of Proposition III.2.12 in \cite{Audin}. In fact, if we take $G=S$ as a connected abelian group $T$ so that $C(T)=T$, any $T$-orbit is contained in a level set of the moment map $\mu_Z=\mu$. Conversely, if any $G$-orbit is contained in a level set of $\mu$ of $G$ and $G\curvearrowright M$ is effective, then $G$ is abelian. 
  }
\end{remark}

Denote the identity component of $Z_G(S)$ and its moment map by $Z_G(S)^0$ and $\mu_Z: M\rightarrow \mathfrak{z_g(s)}^*$, respectively. If $Z'$ is a subgroup of $Z_G(S)^0$, then it is easy to see that $\mu^{-1}_{Z'}(c)\supseteq \mu^{-1}_Z(c)$ for any $c\in (\mathfrak{z}')^*\subseteq \mathfrak{z_g(s)}^*$. In particular, $\mu^{-1}_Z(0)$ is the smallest $S$-invariant $0$-level set of the moment map among the actions of subgroups of $Z_G(S)^0$. Because of this reason, we focus on the $Z_G(S)^0$-action.

Suppose $Z_G(S)^0$ is compact and $c\in \mathfrak{z_g(s)}^*\cap \mu_Z(M)$ is a regular value of $\mu_Z$. Denote the stabilizer subgroup of ${\rm Ad}^*(Z_G(S)^0)\curvearrowright \mathfrak{z_g(s)}^*$ at $c$ by $Z_c:=\{g\in Z_G(S)^0; {\rm Ad}^*(g)c=c\}$. We further assume $Z_c$ acts on $\mu^{-1}_{Z}(c)$ freely. Then, the Marsden-Weinstein-Meyer symplectic reduction yields a symplectic manifold $(M_c:=\mu^{-1}_{Z}(c)/Z_c, \omega_c)$  (cf. \cite{Audin}). We denote the inclusion and the projection by  $\iota: \mu^{-1}_{Z}(c)\rightarrow M$ and $\pi: \mu^{-1}_{Z}(c)\rightarrow M_c$, respectively. 

 By Lemma \ref{lem3.1}, $S$ acts on $\mu^{-1}_{Z}(c)$, and we define the natural action of $S$ on $M_c$ so that  $s\circ \pi=\pi\circ s$ for $s\in S$. Because $Z_c$ is a subgroup of $Z_G(S)$,  the action is well-defined.  
Since $\pi^*{\omega}_c=\iota^*\omega$ and $S\curvearrowright M$ is symplectic, we easily verifies that $S\curvearrowright M_c$ is also symplectic. Moreover, $S\curvearrowright M_c$ is Hamiltonian. In fact, because the restriction  $\mu_S|_{\mu^{-1}_Z(c)}: \mu^{-1}_Z(c)\rightarrow \fs^*$ is  $Z_c$-invariant, we obtain a map $\overline{\mu}_S: M_c\rightarrow \fs^*$ and $\overline{\mu}_S$ is the moment map of the $S$-action on $M_c$.  
% It is obvious that the $S$-action on $M_c$ is transitive if and only if so is the action of $SZ_c$ on $\mu^{-1}_Z(c)$. 

If furthermore, $M$ is K\"ahler and $G$ acts on $M$ as holomorphic isometries,  then  we obtain the K\"ahler reduction $\pi: \mu^{-1}_Z(c)\rightarrow M_c$ in the sense of Theorem 7.2.3 in \cite{Futaki}, and we see the following:
\begin{lemma}\label{lem3.3}
Suppose $M$ is K\"ahler and $G$ acts on $M$ as holomorphic isometries.  Then, $S$ acts on ${M}_c$ as holomorphic isometries.
\end{lemma}

\begin{proof}
We have already shown that $S\curvearrowright M_c$ is symplectic. When $c=0$, by using the structure result of $M_0$ described in Section \ref{Sec2}, one easily verifies that $S\curvearrowright M_0$ is holomorphic since so is $S\curvearrowright M$, $Z_0=Z_G(S)^0$ and $S$ commutes with the $Z_G(S)^0$-action. When $c\neq 0$, we use the shifting trick since $c$ is a general element in $\mathfrak{z}_g(s)^*$.   Set $\mathcal{O}_{-c}:={\rm Ad}^*(Z_G(S)^0)(-c)$. Since $Z_G(S)^0$ is compact, the Kirillov-Kostant-Souriau symplectic form on $\mathcal{O}_{-c}$ is a K\"ahler form so that the inclusion $i: \mathcal{O}_{-c} \rightarrow \mathfrak{z_g(s)}^*$ is the moment map for the $Z_G(S)$-action. Then, $Z_G(S)$ acts on the product K\"ahler manifold $M\times \mathcal{O}_{-c}$ in a Hamiltonian fashion with moment map $\Psi:=\mu_Z+i$. Then, we have an inclusion $\mu^{-1}_Z(c)\rightarrow \Psi^{-1}(0)$ by $p\mapsto (p,c)$ and $\mu^{-1}_Z(c)/Z_c$ is identified with $\Psi^{-1}(0)/Z_G(S)$ by $[p]\mapsto [(p,c)]$ (see \cite{Futaki}).  Moreover, we induce a K\"ahler structure on $\mu^{-1}_Z(c)/Z_c$ from the ones on $\Psi^{-1}(0)/Z_G(S)$. Since the action $S\curvearrowright M\times \mathcal{O}_{-c}$ defined by $s\cdot (p,c)\mapsto (sp,c)$ is holomorphic and commutes with $Z_G(S)^0$-action, we see $S\curvearrowright  \Psi^{-1}(0)/Z_G(S)$ is holomorphic.  This implies the lemma.
\end{proof}

    \begin{example}\label{ex3.4}
  {\rm Let $S$ be the special unitary group $SU(n)$. Consider the $SU(n)$-action on $\mathbb{C}^n$ via the natural representation of $SU(n)$. Note that the center of $SU(n)$ is discrete since $SU(n)$ is semi-simple. A principal $SU(n)$-orbit is the hypersphere $S^{2n-1}(r)$ with radius $r>0$ and the connected component of the centralizer of $S=SU(n)$ in $G=U(n)$ is $Z_G(S)^0=\{e^{\sqrt{-1}\theta} Id_n; \theta\in [0,2\pi]\}\simeq S^1$. One verifies that $S^{2n-1}(r)=\mu^{-1}_{Z}(c)$ for some $c\in \mathfrak{z_g(s)}^*\simeq \mathbb{R}$. The $Z_G(S)^0$-action is nothing but the Hopf action on $\mathbb{C}^n$ and the K\"ahler quotient $M_c$ is the complex projective space $\mathbb{C}P^n$. Moreover,   $SU(n)$ acts on $M_c$ transitively.}
  \end{example}

%%%%%%%Homogeneous Lagrangian submanifolds%%%%%%%%%%%%
%%%%%%%%%%%%%%%%%%%%%%%%%%%%%%%%%%%%%%%%

\subsection{Reductions of homogeneous hypersurfaces}\label{subsec3.2}
In this subsection, we assume $M$ is a simply connected compact symplectic manifold (e.g., a compact K\"ahler-Einstein manifold of positive Ricci curvature) and $G$ is a compact connected Lie group acting on $M$ effectively in a Hamiltonian way.  
We equip $M$ with a $G$-invariant almost K\"ahler structure on $M$. Note that such a structure always exists (see Appendix in \cite{BL}). Throughout this subsection, we denote the Riemannian metric and the almost complex structure by $g$ and $J$, respectively.

 In the following, we further assume $S\curvearrowright M$ is a cohomogeneity one action, namely,  actions such that the principal orbits are real hypersurfaces in $M$. For the $S$-action, we set 
\begin{align*}
\begin{aligned}
M_{pri(S)}&:=\{p\in M; S\cdot p\ {\rm is\ a\ principal\ orbit}\},\\
M_{sing(S)}&:=\{p\in M; S\cdot p\ {\rm is\ a\ singular\ orbit}\},\\
M_{reg(\mu_Z)}&:=\{p\in M; p\ {\rm is\ a\ regular\ point\ of}\ \mu_Z\}\ {\rm and}\\
M_{cri(\mu_Z)}&:=\{p\in M; p\ {\rm is\ a\ critical\ point\ of}\ \mu_Z\}.
\end{aligned}
\end{align*}

We note that, if $M$ is simply connected, then the cohomogeneity one isometric action does not admit any exceptional orbit. Moreover, under the assumption of $M$, the orbit space $M/S$ is homeomorphic to the closed interval $[0,1]$, and there exist exactly two singular orbits $S\cdot p_1$ and $S\cdot p_2$ (cf. Section 2.9.3 in \cite{BCO}). Namely, we have $M_{sing(S)}=S\cdot p_1\sqcup S\cdot p_2$.

 For any real hypersurface $\overline{M}$ in $M$, an almost contact metric structure on  $\overline{M}$ is induced from the almost K\"ahler structure on $M$. For a (local) unit normal vector field $N$ on $\overline{M}$, we define a vector field $\xi\in \Gamma(T\overline{M})$ by $\xi:=-JN$ and call it the {\it Reeb vector field} of $\overline{M}$.

\begin{lemma}\label{lem3.5}
Assume $G$ acts on $M$ effectively in a Hamiltonian way. Let $S$ be a connected closed subgroup of $G$ acting on $M$ as a cohomogeneity one isometric action. If $Z_G(S)$ is not discrete and compact, then we have the following:
\begin{description}
\item[\rm (a)] ${\rm dim}Z_G(S)=1$, namely, the $Z_G(S)^0$-action is a circle action.
\item[\rm (b)] The orbit space $M/S$ is homeomorphic to $\mu_Z(M)$ via $S\cdot p\mapsto \mu_Z(p)$. In particular, we have $M_{pri(S)}=M_{reg(\mu_Z)}$, $M_{sing(S)}=M_{cri(\mu_Z)}$ and $S\cdot p=\mu_Z^{-1}(c)$ with $c=\mu_Z(p)$ for any $p\in M$.
\item[\rm (c)] For any regular value $c\in \mathfrak{z_g(s)^*}$,  $Z_G(S)^0$ acts on  $\mu_Z^{-1}(c)$ freely. Moreover, the $Z_G(S)^0$-action generates an isometric Reeb flow on $\mu_Z^{-1}(c)$ i.e., there exists an element $v\in \mathfrak{z_g(s)}$ such that $\tilde{v}$ is the Reeb vector field on $\mu_Z^{-1}(c)$, and $Z_G(S)^0$-orbits contained in $\mu_Z^{-1}(c)$ are mutually isometric.
\end{description}
\end{lemma}

\begin{proof}
For simplicity, we set $Z:=Z_G(S)^0$ and $\mathfrak{z}:=\mathfrak{z_g(s)^*}$ in this proof.  

Since $\mu_{Z}: M\rightarrow \mathfrak{z}^*$ is a smooth map, Sard's theorem implies $M_{reg(\mu_Z)}$ is a dense subset in $M$. On the other hand, $M_{pri(S)}$ is open dense in $M$. Thus, there exists a point $p_0$ such that $p_0\in M_{reg(\mu_Z)}\cap M_{pri(S)}$. 
Then, $c_0:=\mu_{Z}(p_0)$ is a regular value, and hence, $\mu^{-1}_{Z}(c_0)$ is a submanifold in $M$ with ${\rm codim} \mu^{-1}_{Z}(c_0)={\rm dim}\mathfrak{z^*}={\rm dim}Z$. Since $S\cdot p_0$ is a submanifold contained in $\mu^{-1}_{Z}(c_0)$ by Lemma \ref{lem3.1} and $S\curvearrowright M$ is cohomogeneity one, we see ${\rm dim}Z={\rm codim} \mu^{-1}_{Z}(c_0)\leq {\rm codim} S\cdot p_0=1$. Because $Z$ is not discrete, we conclude ${\rm dim} Z={\rm codim} \mu^{-1}_{Z}(c_0)=1$. Thus, $Z\curvearrowright M$ is a circle action by the compactness of $Z$.
This proves (a).

 Next, we shall show (b). 
For any $p\in M_{reg(\mu_Z)}\cap M_{pri(S)}$, $c=\mu_Z(p)$ is a regular value and $\mu_Z^{-1}(c)$ is a hypersurface in $M$ by (a). Therefore, the regular orbit $S\cdot p$ is a connected component of $\mu_Z^{-1}(c)$ since $S\cdot p$ is a connected open subset in $\mu_Z^{-1}(c)$ and the action is proper. On the other hand, it is known that $\mu^{-1}_{Z}(c)$ is a connected subset in $M$ by the compactness of $M$ and $Z$ (see \cite{GS}). Thus, we see $S\cdot p=\mu_Z^{-1}(c)$ for any $p\in M_{reg(\mu_Z)}\cap M_{pri(S)}$. 

We claim that $M_{reg(\mu_Z)}\cap M_{sing(S)}=\{\phi\}$. Suppose the contrary were true. Take a point $p_1\in M_{reg(\mu)}\cap M_{sing(S)}$ and set $c_1:=\mu_Z(p_1)$. Then,  we have (i): $\mu^{-1}_{Z}(c_1)$ is a $S$-invariant connected hypersurface, (ii): ${\mu}^{-1}_Z(c_1)$ consists of singular orbits of $S\curvearrowright M$, namely, $\mu_Z^{-1}(c_1)\cap M_{reg(S)}=\{\phi\}$ (otherwise, we have $S\cdot p=\mu_Z^{-1}(c_1)\varsupsetneq S\cdot p_1$ for some $p\in \mu_Z^{-1}(c_1)\cap M_{pri(S)}$, a contradiction).  
In particular, ${\rm dim}(\mu_Z^{-1}(c_1)/S)\geq 1$ and there are infinitely many singular orbits. This contradicts to the general fact that the cohomogeneity one action admits at most two singular orbit.  Therefore, we have $M_{reg(\mu_Z)}\cap M_{sing(S)}=\{\phi\}$. Notice that this implies $M_{reg(\mu_Z)}=M_{pri(S)}$ and $M_{cri(\mu_Z)}=M_{sing(S)}$.  

Since $Z\simeq S^1$, the image $\mu_Z(M)$ is homeomorphic to a closed interval $I$ in $\mathbb{R}\simeq \mathfrak{z^*}$ by the convexity theorem for the moment map (see Corollary IV.4.5 in \cite{Audin}).  Since $M_{cri(\mu_Z)}=M_{sing(S)}=S\cdot p_1\sqcup S\cdot p_2$ and $\mu_Z(S\cdot p_i)\equiv c_i$ for $i=1,2$ by Lemma \ref{lem3.1}, there are exactly two critical values, the maximum  and the minimum  of $\mu_Z$, and we see $S\cdot p_i=\mu_Z^{-1}(c_{i})$. Therefore, the map $M/S\rightarrow \mu_Z(M)$ via $S\cdot p\mapsto \mu_Z(p)$ defines a homeomorphism. This proves (b).

 Finally, we show (c). By (a), $Z$ is an abelian group, and hence, $Z$ acts on each level set $\mu_Z^{-1}(c)$. For $p\in \mu_Z^{-1}(c)$, we denote the stabilizer subgroup at $p$ of the action $Z\curvearrowright M$ and its Lie subalgebra by $Z_p$ and $\mathfrak{z}_p$, respectively. It is a general fact for the moment map that ${\rm Im}d\mu_Z(p)$ coincides with the annihilator in $\mathfrak{z^*}$ of $\mathfrak{z}_p$ (see Proposition III.2.2 in \cite{Audin}). Thus,    
 \begin{itemize}
\item if $p\in M_{cri(\mu_Z)}$, then we see $\mathfrak{z}_p=\mathfrak{z}$ since $Z\simeq S^1$, and hence, $p$ is a fixed point of the $Z$-action, and
\item  if $p\in M_{reg(\mu_Z)}$, then the action $Z\curvearrowright \mu_Z^{-1}(c)$, where $c=\mu_Z(p)$, is locally free, i.e., the stabilizer subgroup $Z_p$ of $Z\curvearrowright \mu_Z^{-1}(c)$ at $p$ is discrete.  
\end{itemize}
We claim that $Z_p=Z_q$ for any $p,q\in M_{reg(\mu_Z)}=M_{pri(S)}$.

 Setting $S':=SZ=SZ_G(S)^0$, we see $S'$ is a subgroup of $G$ and $S'\cdot p=S\cdot p$ for any $p\in M$ since $Z$ acts on each $S$-orbit by (b). Note that we have $M_{pri(S')}=M_{pri(S)}$. The $S'$-action induces an action of the subgroup $Z$ on the homogeneous space $S'/S'_p\simeq S'\cdot p$.   Since $S'\cdot p=S\cdot p=\mu_Z^{-1}(c)$ for any $p\in M$ by (b), the action $Z\curvearrowright \mu_Z^{-1}(c)$ is equivariant to $Z\curvearrowright S'/S'_p$. Take distinct points $p,q \in M_{pri(S')}$. Then, $h\in Z_p$ if and only if $h\in S'_p$. Since $S'_q$ is conjugate to $S'_p$ in $S'$, there exists an element $s'\in S'$ such that $s'h(s')^{-1}\in S'_q$. On the other hand, because $Z$ is abelian, $h$ commutes with any element in $S'=SZ$. Therefore, we see $h=s'h(s')^{-1}\in S'_q$, i.e., $h\in Z_q$. This implies $Z_p=Z_q$ for any $p, q\in M_{pri(S')}=M_{pri(S)}=M_{reg(\mu_Z)}$ as claimed. 
 
 In particular, an element $h\in Z_p$ for $p\in M_{reg(\mu_Z)}$ fixes every point in $M_{reg(\mu_Z)}$, and hence, whole $M$.
On the other hand, because $G\curvearrowright M$ is effective, $Z\curvearrowright M$ is also an effective action. Therefore, $Z_p=\{e\}$ for $p\in M_{reg(\mu_Z)}$ and the action $Z\curvearrowright \mu_Z^{-1}(c)$ is free for any regular value $c$.
 
 Take an element $v \in \mathfrak{z}$ so that $|\tilde{v}_p|_g=1$ for some point $p\in \mu_Z^{-1}(c)$, where $c$ is a regular value. Since $S$ acts on $\mu_Z^{-1}(c)$ transitively, there exists $s\in S$ such that $q=sp$ for any $q\in \mu_Z^{-1}(c)$. Then,  $|\tilde{v}_q|_g=|\widetilde{({\rm Ad}(s^{-1})v)}_p|_g=|\tilde{v}_p|_g=1$ because the metric is $G$-invariant and $v\in \mathfrak{z}=\mathfrak{z_g(s)}$.  On the other hand,  we have $T_q\mu_Z^{-1}(c)=\{\tilde{X}_q;\ X\in \fs\}$ by (b). Then, we see
   \begin{align*}
   g(J\tilde{v}, \tilde X)=\omega(\tilde{v}, \tilde X)=d\mu^{v}_Z(\tilde{X})=\tilde{X}\mu^{v}_Z=0
   \end{align*}
   for any $X\in \fs$. Therefore, $J\tilde{v}$ is a unit normal vector  field of $\mu_Z^{-1}(c)$, i.e., $\tilde{v}$ is the Reeb vector field. Moreover, for any $q=sp$, we have $Z\cdot q=Z\cdot sp=s(Z\cdot p)$. Thus, $Z$-orbits in $\mu_Z^{-1}(c)$ are mutually isometric.
  \end{proof}

By Lemma \ref{lem3.5} (a), we may assume $Z_G(S)^0\simeq S^1$. Consider the moment map $\mu_{{Z}}: M\rightarrow \mathfrak{z_g(s)}^*$ of the $Z_G(S)^0$-action.  In the following, we fix an inner product $\langle, \rangle$ on $\mathfrak{z_g(s)}$ and identify $\mu_Z$ with a real function $\mu_Z^{X_0}: M\rightarrow \mathbb{R}$ for a unit vector $X_0\in \mathfrak{z_g(s)^*}$. Since $\tilde{X_0}$ generates an $S^1$-orbit, it turns out that the hamiltonian $\mu_Z^{X_0}$ is a Morse-Bott function (see Theorem IV.2.3 in \cite{Audin}). In our setting, much more is true:

Recall that a smooth function $f: M\rightarrow \mathbb{R}$ on a Riemannian manifold $M$ is called {\it transnormal} if $|\nabla f|^2_g=a\circ f$ for some smooth function $a: I\rightarrow \mathbb{R}$, where $I=f(M)$. Furthermore, a transnormal function $f$ is called {\it isoparametric} if there exists a continuous function $b: I\rightarrow \mathbb{R}$ such that $\Delta f=b\circ f$. A regular level set of an isoparametric function is so called the {\it isoparametric hypersurface}. 

\begin{proposition}\label{prop3.6}
$\mu_Z:M\rightarrow \mathfrak{z_g(s)^*}\simeq \mathbb{R}$ is an isoparametric function.
\end{proposition}
\begin{proof}
First, we shall show $\mu_Z$ is transnormal. We set $Z:=Z_G(S)^0$ and $\mathfrak{z}:=\mathfrak{z_g(s)^*}$.  By Lemma \ref{lem3.5} (c), $Z$ acts on $M_{reg(\mu_Z)}$ freely, and hence, each orbit $\mathcal{O}_p:=Z\cdot p$ through $p\in M_{reg(\mu_Z)}$ is a principal $Z$-orbit. Moreover, since the $Z$-action fixes every point in $M_{cri(\mu_Z)}$ (see the proof of Lemma \ref{lem3.5} (c)), we have $M_{reg(\mu_Z)}=M_{pri(Z)}$ and  $M_{cri(\mu_Z)}=M_{sing(Z)}$.

Thus, we define the volume function on $M_{pri(Z)}$ as well as \eqref{eq3} by
\begin{align}\label{eq27}
V(p):={vol}_g(\mathcal{O}_p)=|\nu|_g(p)\int_{Z}\nu.
\end{align}
It is known that  $V$ is a smooth function on $M_{pri(Z)}$. Moreover, $V$ can be  extended continuously to the singular sets $M_{sing(Z)}$ of the $Z$-action as $V(p)=0$ for $p\in M_{sing(Z)}$ and $V^2$ is a smooth function on $M$ (see Proposition 1 in \cite{Pacini}).

 Since $\mathcal{O}_p\simeq Z$ for any $p\in M_{pri(Z)}$, we have an embedding $\psi_p: Z\rightarrow M$ so that $\psi_p(Z)=\mathcal{O}_p$ for each $p\in M_{pri(Z)}$. Then, the induced metric $\psi_p^*g$ is left invariant, and hence, it defines another inner product $\langle,\rangle_p$ on $\mathfrak{z}$ so that $\langle X, Y\rangle_p=g(\tilde{X}_p, \tilde{Y}_p)$ for any $X,Y\in \mathfrak{z}$. Because $\mathfrak{z}\simeq \mathbb{R}$, there exists a positive constant $A(p)$ depending on $p$ such that  $\langle, \rangle_p=A(p)\langle, \rangle$.  
Since $\nabla\mu^X_Z=J\tilde{X}$ for any $X\in \fk$, we see
\begin{align}\label{eq28}
\langle X,X\rangle_p=|\tilde{X}|^2_g(p)=|\nabla\mu^X_Z|^2_g(p)=|\nabla\mu_Z|^2_g(p)\cdot \langle X,X\rangle,
\end{align}
and hence, $A(p)=|\nabla\mu_Z|^2_g(p)$.

Define a left invariant metric $h$ on $Z$ by using $\langle, \rangle$ so that $\psi^*_pg=A(p)h$ on $Z$. Then, the volume elements $dv_{\psi^*_pg}$ and $dv_h$ on $Z$ defined by $\psi^*_pg$ and $h$, respectively,  satisfy $dv_{\psi^*_pg}=A(p)^{1/2}dv_h$. Since $V(p)=\int_Zdv_{\psi^*_pg}$, we see
\begin{align*}
A(p)=V(p)^2\cdot (const.).
\end{align*}
for any $p\in M_{pri(Z)}$, where the constant is non-zero.   Therefore, $A$ can be extended to a smooth function on $M$ as $A(p)=0$ for $p\in M_{sing(Z)}$ and satisfies $A(p)=|\nabla\mu_Z|^2_g(p)$ for any $p\in M$. 

By Lemma \ref{lem3.5} (c), the $Z$-orbits contained in ${\mu}_{Z}^{-1}(c)$ are mutually isometric for each regular value $c\in \mathfrak{z}^*$. Thus,  we obtain a well-defined function $a: I=\mu_Z(M)\rightarrow \mathbb{R}$ by $a(c):=A(p)=V(p)^2$ for $p\in {\mu}_{Z}^{-1}(c)$ so that $a$ satisfies $|\nabla \mu_Z|^2_g=a\circ \mu_Z$. Since $M_{reg(\mu_Z)}=M_{pri(S)}$ (see Lemma \ref{lem3.5} (b)) and $\mu_Z|_{M_{reg(\mu_Z)}}: M_{reg(\mu_Z)}\rightarrow \mathfrak{z}^*$ is a submersion, we see $a$ is smooth on $\mu_Z(M_{reg(\mu_Z)})=I^i$, the interior of $I=\mu_Z(M)$. Moreover, for a critical value $c_i\in \partial I$,  $\mu_Z^{-1}(c_i)$ coincides with a singular orbit, and hence, one can easily verify that $a$ is extended to a smooth function on an open interval $I'$ containing $I$ since $A$ is smooth on $M$. Namely, $a$ is differentiable on the boundary. Therefore, $\mu_Z$ is a transnormal function. 

By Lemma \ref{lem3.5} (b), each regular level set of the transnormal function $\mu_Z$ coincides with a principal $S$-orbit. In particular, each level set has constant mean curvature. Then, $\mu_Z$ is isoparametric by Proposition 2.9.1 in \cite{BCO}.
\end{proof}

It is known that any homogeneous hypersurface in a complete Riemannian manifold is isoparametric. In our setting, by Lemma \ref{lem3.5} (b), ${\mu}_Z$ gives an isoparametric function for the $S$-orbits.

Now, we further assume $M$ is a closed K\"ahler-Einstein manifold of positive Ricci curvature and $S$ is a connected closed subgroup of $G={\rm Aut}(M,\omega, J)^0$. Suppose $S$ acts on $M$ as  a cohomogeneity one action and $Z_G(S)$  is not discrete. Then, we have the canonical moment map of the $Z_G(S)^0$-action $\tilde{\mu}_Z: M\rightarrow \mathfrak{z_g(s)}^*\simeq \mathbb{R}$, and  $\tilde{\mu}_Z$ is regarded as an isoparametric function by Proposition \ref{prop3.6}. In fact, $\tilde{\mu}_Z$ satisfies $\Delta \tilde{\mu}_Z=2C\tilde{\mu}_Z$ (see Proposition \ref{prop2.7} and Remark \ref{rmk2.8}). Moreover, since $M$ is closed, we have
\begin{align*}
\int_M\tilde{\mu}_{Z}\omega^n=0.
\end{align*}
This implies $0$ is an interior point in the closed interval $I=\tilde{\mu}_Z(M)$. Because there is no critical value in the interior $I^i$ (see the proof of Lemma \ref{lem3.5} (b)), $0$ is a regular value of $\tilde{\mu}_Z$. Therefore, the $0$-level set $\tilde{\mu}_Z^{-1}(0)$ is a homogeneous hypersurface in $M$ and $S$ acts on $\tilde{\mu}_Z^{-1}(0)$ transitively by Lemma \ref{lem3.5} (b).  Denote the stabilizer subgroup of $S$ at $p\in \tilde{\mu}^{-1}_Z(0)$ by $S_p$. Then, $\tilde{\mu}^{-1}_Z(0)\simeq S/S_p$. It is easy to verify that $S$ also acts on the quotient space $M_0=\tilde{\mu}^{-1}_Z(0)/Z_G(S)^0$ transitively, and the stabilizer subgroup of $S$ at $[p]\in M_0$ is given by $Z_G(S)^0S_p$. Note that $S$ acts on $M_0$ as holomorphic isometries by Lemma \ref{lem3.3}.  Moreover,  by Lemma \ref{lem3.5} (c) and Proposition \ref{prop2.9} (see also Remark \ref{rmk2.10} (1)), $(M_0, \omega_0, J_0)$ is a K\"ahler-Einstein manifold of positive Ricci curvature.

Summing up our arguments, we prove the following:
\begin{theorem}\label{thm3.7}
Let $(M,\omega, J)$ be a closed K\"ahler-Einstein manifold of positive Ricci curvature. Suppose a connected compact subgroup $S$ of $G={\rm Aut}(M, \omega, J)^0$ acts on $M$ as a cohomogeneity one action. Furthermore, we assume the centralizer $Z_G(S)$ of $S$ in $G$ is not discrete. Then, the $Z_G(S)^0$-action is a circle action and the canonical moment map $\tilde{\mu}_Z: M\rightarrow \mathfrak{z_g(s)^*}\simeq \mathbb{R}$ of the $Z_G(S)^0$-action is an isoparametric function for the $S$-orbits. 

Moreover, $0$ is a regular value of $\tilde{\mu}_Z$ and $Z_G(S)^0$ acts on $\tilde{\mu}^{-1}_Z(0)$ freely, and the K\"ahler quotient space $(M_0=\tilde{\mu}^{-1}_Z(0)/Z_G(S)^0, \omega_0, J_0)$ is a compact homogeneous K\"ahler-Einstein manifold of positive Ricci curvature which is diffeomorphic to $S/Z_G(S)^0S_p$, where $S_p$ is the stabilizer subgroup of $S$ at $p\in \tilde{\mu}^{-1}_Z(0)$. 
\end{theorem}

Since the $Z_G(S)^0$-orbits contained in $\tilde{\mu}^{-1}_Z(0)$ are mutually isometric, the Hsiang-Lawson metric $g_{HL}$ coincides with a constant multiple of the reduced K\"ahler metric $g_0$. Therefore, combining Theorem \ref{thm3.7} with Theorem \ref{thm2.11} and Dong's result in \cite{Dong}, we obtain
\begin{corollary}\label{cor3.8}
Suppose the same assumptions described in Theorem \ref{thm3.7}. Then, any minimal (resp. Hamiltonian minimal) Lagrangian submanifold $L_0$ in the compact homogeneous K\"ahler-Einstein manifold $(M_0, \omega_0, J_0, g_0)$ yields a $Z_G(S)^0$-invariant minimal (resp. Hamiltonian minimal) Lagrangian submanifold in $(M, \omega, J, g)$ as the pre-image $\pi^{-1}(L_0)$ of the fibration $\pi: \tilde{\mu}^{-1}(0)\rightarrow M_0$. Conversely, any $Z_G(S)^0$-invariant minimal Lagrangian submanifold in $(M, \omega, J, g)$ is obtained in this way.
\end{corollary}

It is known that a compact homogeneous K\"ahler-Einstein manifold of positive Ricci curvature is simply connected and is obtained by a K\"ahler product of generalized flag manifolds (see Section 8 in \cite{Besse}). Thus, the quotient space obtained in Theorem \ref{thm3.7} may be well-understood.  For instance, a totally geodesic (and hence, automatically minimal) Lagrangian submanifold in a generalized flag manifold $M_0$ is obtained by a real form of $M_0$. See \cite{GP} for the construction of the real forms.

\begin{remark}\label{rmk3.9}
{\rm We remark other properties of the K\"ahler reduction given in Theorem \ref{thm3.7}. 

(i) $\pi: \tilde{\mu}^{-1}(0)\rightarrow M_0$ preserves the homogeneity of a submanifold in some sense. Namely, if $L_0$ is a homogeneous submanifold in $M_0$ obtained by an orbit of a connected subgroup $S'$ of $S$, then so is the pre-image $L=\pi^{-1}(L_0)$ in $M$. In fact, $S'Z_G(S)$ acts on $L$ transitively. 

(ii) Since the $Z_G(S)^0$-orbits contained in $\tilde{\mu}^{-1}_Z(0)$ are mutually isometric, we have $vol_g(L)=(const)\cdot vol_{g_0}(L_0)$ for any $Z_G(S)^0$-invariant Lagrangian submanifold $L$. Thus, applying Lemma 3.1.1 in \cite{LR}, we see that if $L$ is Hamiltonian stable with respect to $g$, then so is $L_0$ with respect to $g_0$.}
\end{remark} 

\subsection{Examples}\label{subsec3.3}

The condition (c) in Lemma \ref{lem3.5} implies the situation is somewhat restrictive.  So far, we know some classification results of real hypersurfaces with isometric Reeb flow in Hermitian symmetric spaces (\cite{BS1}, \cite{BS2} and \cite{Okumura}). These examples are obtained by a Hermann action of cohomogeneity one.  

Let $(G, S)$ and $(G, S')$ be Hermitian symmetric paris of compact type (see \cite{Hel} for the details). Then, $S$ acts on the Hermitian symmetric space $M=G/S'$ as holomorphic isometries, and we call such an action the {\it Hermann action of Hermitian type}. %なぜ?
 If $G/S'$ is an irreducible Hermitian symmetric space and $S\curvearrowright G/S'$ is cohomogeneity one, then  these actions are preciously given in Table 1. 
 
 Since $(G, S)$ is a Hermitian symmetric pair, $S$ has a 1-dimensional center $C(S)$. Thus, the centralizer $Z_G(S)$ is not discrete because $C(S)\subseteq Z_G(S)$. 
 
 Suppose $S\curvearrowright G/S'$ is cohomogeneity one. Then, ${\rm dim}Z_G(S)=1$ by Lemma \ref{lem3.5} (a), and hence, $Z_G(S)^0=C(S)^0$.  Moreover, we may assume $G={\rm Aut}(M,\omega, J)^0$.
Thus, $C(S)^0$ acts on $\tilde{\mu}_Z^{-1}(0)$ freely by Lemma \ref{lem3.5} (c). In particular, we can apply  Theorem \ref{thm3.7} and Corollary \ref{cor3.8}.
 
\begin{table}[htb]
\center{
  \begin{tabular}{|c|c|}\hline
$S$ & {Hermitian symmetric space $M=G/S'$} \\ \hline
 $S(U(k+1)\times U(n-k))$ & $SU(n+1)/S(U(n)\times U(1))\simeq\mathbb{C}P^{n}$   \\ \hline
  $S(U(n)\times U(1))$ & $SU(n+1)/S(U(k+1)\times U(n-k))\simeq \tilde{G}_{k+1}(\mathbb{C}^{n+1})$ \\ \hline
$U(n)$ & $SO(2n)/SO(2)\times SO(2n-2)\simeq \tilde{G}_2(\mathbb{R}^{2n})$ \\ \hline
$SO(2)\times SO(2n-2)$ &  $SO(2n)/U(n)$ \\ \hline
 \end{tabular}
 \caption {Hermann actions of Hermitian type of cohomogeneity one acting on an irreducible Hermitian symmetric space, where  $0\leq k \leq n-1$.}
}
\end{table}

We shall give further details of the K\"ahler quotients for some cases in which we obtain several examples of minimal Lagrangian submanifolds:
\begin{example}\label{ex3.10}
{\rm 
 Let $M$ be the complex projective space $\mathbb{C}P^n$ of constant holomorphic curvature 4. By the result of Okumura \cite{Okumura}, a real hypersurface $\overline{M}$ in $\mathbb{C}P^n$ admits an isometric Reeb flow if and only if  $\overline{M}$ is an open part of a tube around a totally geodesic $\mathbb{C}P^k\subset \mathbb{C}P^n$ for $k=0,\ldots, n-1$. The tube is a homogeneous hypersurface in $\mathbb{C}P^n$, and $S=S(U(k+1)\times U(n-k))$ acts on it transitively.
 The center is given by 
 \begin{align*}
 C(S)=\{{\rm diag}(e^{\sqrt{-1}(n-k)\theta_1}Id_{k+1}, e^{-\sqrt{-1}(k+1)\theta_1}Id_{n-k}); \theta_1\in \mathbb{R}\}\simeq S^1.
\end{align*}

We further restrict our attention to the case when $k=0$, i.e., $S$ is the stabilizer subgroup $S'=S(U(1)\times U(n))$. In this case, the regular orbit $\overline{M}$ is a geodesic hypersphere in $\mathbb{C}P^{n}$. It is known that the geodesic hypersphere is characterized by  the {\it totally $\eta$-umbilic} hypersurface, namely, the shape operator $A$ of $\overline{M}$ satisfies $A(Z)=aZ+b\xi^{\sharp}(Z)\xi$ for any $Z\in \Gamma (T\overline{M})$, where $\xi$ is the characteristic vector field, and $a,b\in \mathbb{R}$. In our setting, $\xi$ is given by a unit length vector field $\tilde{v}$ on $\overline{M}$ generated by the $C(S)$-action. In particular,  we have $A(X')=aX'$ for any $X'\in E_p$ and $p\in \overline{M}$. Therefore, by the theorem of Kobayashi \cite{Kobayashi}, the holomorphic sectional curvature $K_0$ of the quotient space $M_0$ is related to the holomorphic sectional curvature $K$ of $M$ by
\begin{align*}
K_0(X)=K(X')+4g(A(X'), X')^2=4+4a^2
\end{align*}
for any unit vector $X\in T_xM_0$. In particular, we see $M_0$ is a compact K\"ahler manifold of positive constant holomorphic sectional curvature.
Moreover, $M_0$ is simply connected since $M_0$ is positive compact K\"ahler-Einstein. Thus, $M_0$ is holomorphic isometric to $\mathbb{C}P^{n-1}$. 

Therefore, by Corollary \ref{cor3.8}, we see that  {any  $C(S')$-invariant minimal Lagrangian submanifold in $\mathbb{C}P^{n}$ corresponds to a minimal  Lagrangian submanifold in $\mathbb{C}P^{n-1}$ via the fibration $\pi: \overline{M}=\tilde{\mu}^{-1}(0)\rightarrow M_0\simeq \mathbb{C}P^{n-1}$. }
We note that many examples of minimal Lagrangian submanifold in $\mathbb{C}P^{n-1}$ are known (e.g. \cite{BG}), and these examples yield minimal Lagrangian submanifolds in $\mathbb{C}P^{n}$.
}
\end{example}

\begin{example}\label{ex3.11}
{\rm Let $M$ be the oriented real two-plane Grassmannian manifold $SO(2n)/(SO(2)\times SO(2n-2))\simeq \tilde{G}_2(\mathbb{R}^{2n})$. $\tilde{G}_2(\mathbb{R}^{2n})$ can be identified with the complex hyperquadric $Q^{2n-2}(\mathbb{C})$, where $2n-2={\rm dim}_{\mathbb{C}}\tilde{G}_2(\mathbb{R}^{2n})$, and the real hypersurfaces with  isometric Reeb flows in $Q^{2n-2}(\mathbb{C})$ (or more generally, in complex hyperquadrics of arbitrary dimensions) were classified by Berndt and Suh (\cite{BS2}). Such a real hypersurface $\overline{M}$ is obtained by an orbit of the unitary group $U(n)$ as a closed subgroup of $SO(2n)$, and  $\overline{M}\simeq U(n)/(U(1)\times U(n-2))$ (See \cite{BS2} for the details). In particular, the K\"ahler quotient space $M_0$ is diffeomorphic to the homogeneous space $U(n)/S^1\cdot (U(1)\times U(n-2))$, where $S^1$ is the center of $U(n)$. 

Let us consider an oriented real hypersurface $\overline{L}$ in the odd-dimensional unit sphere $S^{2n-1}(1)\subset \mathbb{R}^{2n}$. Define the Gauss map $\mathcal{G}: \overline{L}\rightarrow \tilde{G}_2(\mathbb{R}^{2n})\simeq Q^{2n-2}(\mathbb{C})$ via $p\mapsto V_p$, where $V_p$ is the oriented normal space of $\overline{L}$ in $S^{2n-1}(1)$ regarded as a two-plane in $\mathbb{R}^{2n}$. It is known that $\mathcal{G}$ is a Lagrangian immersion into $Q^{2n-2}(\mathbb{C})$, and hence we obtain a Lagrangian submanifold $L$ as the Gauss image $\mathcal{G}(\overline{L})$.  Moreover, if $\overline{L}$ is an isoparametric hypersurface  in $S^{2n-1}(1)$, then  $L=\mathcal{G}(\overline{L})$ is a minimal Lagrangian submanifold in $Q^{2n-2}(\mathbb{C})$ (see \cite{MO}). Since the action $U(n)\curvearrowright Q^{2n-2}(\mathbb{C})$ is equivariant to the natural action $U(n)\curvearrowright \mathbb{R}^{2n}\simeq \mathbb{C}^n$ via $\mathcal{G}$, %ここチェック
if $\overline{L}$ is $S^1$-invariant, then so is ${L}$. Therefore, we obtain $S^1$-invariant minimal Lagrangian submanifold  in $Q^{2n-2}(\mathbb{C})$ from a $S^1$-invariant isoparametric hypersurface in $S^{2n-1}(1)$, and such a hypersurface comes from an isoparametric hypersurface in $\mathbb{C}P^{n-1}$ (see \cite{Do}). Moreover, explicit examples of isoparametric hypersurfaces in $\mathbb{C}P^{n-1}$ are found in \cite{BCO} and \cite{Do}.  Thus, by Theorem \ref{thm2.11} and \ref{thm3.7}, we obtain some examples of minimal Lagrangian submanifold in the K\"ahler quotient space $M_0$ from isoparametric hypersurfaces in $\mathbb{C}P^{n-1}$.
 }
\end{example}

{\bf Acknowledgements.}
The author would like to thank J\"urgen Berndt and Yoshihiro Ohnita for some suggestions. He also thanks to Takahiro Hashinaga for helpful discussions and Anna Gori for sharing a result of joint work with Lucio Bedulli about the formula of moment map.  A part of this work was done while the author was staying at King's College London and University of T\"ubingen by the JSPS Program for Advancing Strategic International Networks to Accelerate the Circulation of Talented Researchers, Mathematical Science of Symmetry, Topology and Moduli, Evolution of International Research Network based on OCAMI. He is grateful for the hospitalities of the college and the university.

%\begin{acknowledgements}
%\end{acknowledgements}



\end{document}